\documentclass[11pt]{amsart}
\usepackage{amssymb, amsmath, amsthm, mathrsfs, amsopn, latexsym}
\usepackage{amsrefs}

\usepackage{graphicx}
\usepackage{color}
\usepackage[a4paper, centering]{geometry}
\DeclareMathOperator{\Tr}{Tr}
\newcommand*{\Trace}[3][\pm]{\Tr^{#1}(#2, #3)}
\newcommand*{\Trp}[2]{\Trace[i]{#1}{#2}}
\newcommand*{\Trm}[2]{\Trace[e]{#1}{#2}}
\newcommand*{\chiut}[1][]{\chi^{#1}_{\{u > t\}}}
\newcommand*{\jump}[1]{\Theta_{#1}}

\def\nuint{\widetilde{\nu}}

\newcommand{\A}{\boldsymbol{A}}

\newcommand{\mset}{\omega}

\def\DM{{\mathcal{DM}^{\infty}}}
\newcommand*{\BVA}[1][\Omega]{BV({#1}) \cap L^1(#1, |\Div\A|)}
\newcommand*{\BVL}[1][\R^N]{BV(#1)\cap L^{\infty}{(#1)}}
\newcommand*{\BVAlocloc}[1][\Omega]{BV_{\rm loc}({#1})\cap L^1_{\rm loc}(#1, |\Div\A|)}
\newcommand*{\DMlocloc}[1][\Omega]{\mathcal{DM}^{\infty}_{{\rm loc}}{(#1)}}
\newcommand*{\BVLlocloc}[1][\Omega]{BV_{{\rm loc}}(#1)\cap L^{\infty}_{{\rm loc}}{(#1)}}
\newcommand*{\BVAloc}[1][\Omega]{BV({#1})\cap L^1(#1, |\Div\A|)}

\newcommand*{\DMloc}[1][\Omega]{\mathcal{DM}^{\infty}{(#1)}}
\newcommand*{\BVLloc}[1][\Omega]{BV(#1)\cap L^{\infty}{(#1)}}

\newcommand{\mean}[1]{\,-\hskip-1.08em\int_{#1}} 
\newcommand{\Leb}[1]{\mathcal{L}^{#1}} 

\def\R{\mathbb{R}}

\def\N{\mathbb{N}}

\def\cinftio{C^{\infty}_c(\Omega)}
\def\intom#1{\int_{\Omega} #1\,  dx}

\def\cinfc{C^\infty_c(\Omega)}

\DeclareMathOperator{\Div}{div}

\newcommand{\medint}{-\kern  -,375cm\int}
\newcommand{\medintinrigo}{-\kern  -,315cm\int}

\newcommand{\eps}{\varepsilon}
 \newcommand{\hh}{{\mathcal H}^{N-1}}

\newcommand{\LLN}{{\mathcal L}^N}
\newcommand{\LLU}{{\mathcal L}^1}

\newcommand{\M}[1]{\mathcal{#1}}    
\renewcommand{\H}{\M{H}}
\newcommand{\Haus}[1]{{\mathcal H}^{#1}} 

\newcommand{\res}{\mathop{\hbox{\vrule height 7pt width .5pt depth 0pt
\vrule height .5pt width 6pt depth 0pt}}\nolimits} 

\def\pscal#1#2{\left\langle #1\,,\, #2 \right\rangle}

\def\ut{\widetilde{u}}
\def\polar{\theta_\lambda}

\renewcommand{\prec}[2][\lambda]{#2^{#1}}
\newcommand{\pchiut}{\chiut[\lambda]}

\def\intA#1{\int_{\Omega} #1\,  d(\Div\A)}

\newcommand{\pair}[2][\lambda]{\left(#2\right)_{#1}}
\newcommand{\spair}[1]{\pair[*]{#1}}
\def\upiu{u^+}
\def\umeno{u^-}
\def\uint{{u^{i}}}
\def\uext{{u^{e}}}
\def\ust{u^*}
\def\vint{{v^{i}}}
\def\vext{{v^{e}}}

\def\radon{\mathcal{M}(\Omega)}

\def\preciso#1{\widetilde #1}

\long\def\taglio#1{}

\newtheorem{definition}{Definition}[section]

\newtheorem{lemma}[definition]{Lemma}

\newtheorem{theorem}[definition]{Theorem}

\newtheorem{proposition}[definition]{Proposition}

\newtheorem{corollary}[definition]{Corollary}
\theoremstyle{remark}
\newtheorem{remark}[definition]{Remark}
\newtheorem{example}[definition]{Example}

\makeatletter
\def\@settitle{\begin{center}%
		\baselineskip14\p@\relax
		\bfseries
		\uppercasenonmath\@title
		\@title
		\ifx\@subtitle\@empty\else
		\\[5ex]
		\normalsize\mdseries\@subtitle
		\fi
	\end{center}%
}
\def\subtitle#1{\gdef\@subtitle{#1}}
\def\@subtitle{}
\makeatother

\begin{document}
\title[Pairings between divergence--measure vector fields and BV functions]
{Pairings between bounded divergence--measure vector fields and BV functions}

\author[G.~Crasta]{Graziano Crasta}
\address{Dipartimento di Matematica ``G.\ Castelnuovo'', 
Sapienza Universit\`a di Roma\\
	P.le A.\ Moro 5 -- I-00185 Roma (Italy)}
\email{crasta@mat.uniroma1.it}
\author[V.~De Cicco]{Virginia De Cicco}
\address{Dipartimento di Scienze di Base  e Applicate per l'Ingegneria,
 Sapienza Universit\`a di Roma\\
	Via A.\ Scarpa 10 -- I-00185 Roma (Italy)}
\email{virginia.decicco@sbai.uniroma1.it}
\author[A.~Malusa]{Annalisa Malusa}
\address{Dipartimento di Matematica ``G.\ Castelnuovo'', Sapienza Universit\`a di Roma\\
	P.le A.\ Moro 5 -- I-00185 Roma (Italy)}
\email{malusa@mat.uniroma1.it}

\keywords{Divergence--measure vector fields, functions of bounded variation, coarea formula, Gauss-Green formula, semicontinuity}
\subjclass[2010]{26B30,49Q15,49J45}

\date{October 14, 2019}

\begin{abstract}
We introduce a family of pairings between a bounded divergence-measure vector field and a function $u$ of bounded variation, depending
on the choice of the pointwise representative of $u$.
We prove that these pairings inherit from the standard one,
introduced in \cite{Anz,ChenFrid}, all the main
properties and features 
(e.g.\ coarea, Leibniz and Gauss--Green formulas).
We also characterize the pairings making the corresponding
functionals semicontinuous with respect to the strict convergence in $BV$.
We remark that
the standard pairing in general does not share this property.
\end{abstract}

\maketitle

\tableofcontents

\section{Introduction}

In the seminal papers \cite{Anz,ChenFrid},
the product rule
\begin{equation}
\label{f:pair0}
\Div(u \A) = u\, \Div\A + \A\cdot \nabla u\,,
\end{equation}
for smooth functions $u$ and regular vector fields $\A$ in $\R^N$,
has been suitably extended to $BV$ functions and
bounded divergence-measure vector fields.
In particular, Chen and Frid \cite{ChenFrid} showed,
using a regularization argument, that there exists a
finite Radon measure
$\spair{\A, Du}$, which coincides to $\A\cdot\nabla u\, \LLN$
in the smooth case, such that the relation
\begin{equation}\label{f:standard}
\Div(u \A) = u^*\, \Div\A + \spair{\A, Du}
\end{equation}
holds in the sense of measures.
The measure \(\spair{\A, Du}\),
usually called Anzellotti's pairing and
that we call in the sequel 
the \textsl{standard pairing} between $\A$ and $Du$,
is then defined in terms of
the precise representative $u^*$ of $u$,
which is the pointwise value of $u$ obtained as limit
of regularizations by convolutions.

The standard pairing turns out to be a basic tool
in many applications.
We mention here, among others:
extensions of the Gauss--Green formula
\cite{Anz,Cas,ChCoTo,ChTo,ChToZi,ComiMag,ComiPayne,CD3,CD4,LeoSar};
the setting of the Euler--Lagrange equations associated 
with integral functionals defined in $BV$
\cite{ABCM,MaRoSe,Mazon2016};
Dirichlet problems for equations involving the $1$--Laplace operator 
\cite{K1,HI,AVCM,Cas,DeGiSe,DeGiOlPe};
conservation laws
\cite{ChFr1,ChenFrid,ChTo2,ChTo,ChToZi,CD2};
the Prescribed Mean Curvature problem and capillarity
\cite{LeoSar,LeoSar2};
continuum mechanics
\cite{ChCoTo,DGMM,Silh,Schu}.

On the other hand, the standard pairing 
is not adequate when dealing with
obstacle problems in $BV$ (see \cite{SchSch,SchSch2,SchSch3})
or with semicontinuity properties,
as we will explain below.
The aim of this paper is to introduce a new family of pairings,
depending on the choice of the pointwise representative of $u$,
suitable to treat this kind of problems.

\medskip

The main ingredients to build this family of pairings are the
absolute continuity of the measure $\Div\A$ with respect to the
$(N-1)$-dimensional Hausdorff measure $\Haus{N-1}$,
and the fact that the pointwise value of a $BV$ function
can be specified up to a $\Haus{N-1}$-negligible set.
Indeed, a $BV$ function $u$ is approximately continuous outside
a singular set $S_u$ and
its approximate upper and lower limits $u^+$ and $u^-$
coincide with the traces of $u$ on the
countably $\Haus{N-1}$-rectifiable
jump set $J_u\subset S_u$, with $\Haus{N-1}(S_u\setminus J_u) = 0$
(see Section~\ref{ss:BV}).
Hence, a representative of $u$ can be defined
by its approximate limit $\tilde{u}$ outside $S_u$ and
through its traces $u^\pm$ on $J_u$.
We remark again that the presence of $u^* := (u^++u^-)/2$ 
in \eqref{f:standard}
as the pointwise representative of $u$
is due to the regularization argument used in \cite{ChenFrid}
in order to define the standard pairing.

Recently, Scheven and Schmidt \cite{SchSch,SchSch2,SchSch3}
have been in need to introduce
the pairing 
\begin{equation}
\label{f:sspair}
\pair[1]{\A, Du} := -u^+ \Div\A + \Div(u\A)
\end{equation}
in order to study weakly $1$-superharmonic functions
and minimization problems for the total variation
with an obstacle.
Indeed, in this case, the presence of the representative $u^+$
comes out 
from \eqref{f:pair0}
using the one-sided approximation procedure of $u$
introduced in~\cite{CDLP}.

\medskip
In this paper we prove that,
for every Borel function $\lambda\colon\R^N\to [0,1]$,
there exists a measure
$\pair{\A, Du}$ such that
\begin{equation}\label{f:genp}
\Div(u \A) = \prec{u}\, \Div\A + \pair{\A, Du},
\end{equation}
where
$\prec{u} := (1-\lambda)\umeno+\lambda\upiu$
is a selection of the multifunction $x\mapsto [\umeno(x), \upiu(x)]$.
We show that, if the jump part $\Div^j\A$ of $\Div\A$ vanishes
(see Proposition~\ref{p:basicVF} for the definition),
then $\pair{\A, Du}$ is independent
of $\lambda$.

We show that this freedom in the choice of $\prec{u}$ 
is necessary in order to obtain semicontinuity results
in $BV$ for the functionals 
\begin{equation}
\label{f:Fu}
F_\varphi(u) := \pscal{\pair{\A, Du}}{\varphi},
\qquad
\varphi\in C_c(\R^N),\quad
\varphi\geq 0\,.
\end{equation}
We characterize
the selections $\lambda$ such that these functionals
are lower (resp.\ upper) semicontinuous with respect to
the strict convergence in $BV$.
More precisely, 
denoting by $(\Div\A)^\pm$ the 
positive and the negative part of the measure $\Div\A$,
the choices of $\lambda$ which guarantee the lower
semicontinuity of the functionals in \eqref{f:Fu} satisfy
\begin{equation}
\label{f:lsc0}
\pair{\A, Du} = -u^+\, (\Div\A)^+ + u^-\, (\Div\A)^- + \Div(u\A),
\end{equation}
whereas the upper semicontinuity is characterized by
\begin{equation}
\label{f:usc0}
\pair{\A, Du} = -u^-\, (\Div\A)^+ + u^+\, (\Div\A)^- + \Div(u\A)\,.
\end{equation}
As a consequence, it is a matter of fact that, in general, the standard pairing
does not share these semicontinuity properties.
On the other hand, if $\Div\A \leq 0$, as in \cite{SchSch,SchSch2,SchSch3},
from the above result follows that the pairing \eqref{f:sspair} is upper semicontinuous
with respect to the strict convergence in $BV$.

\medskip
The plan of the paper is the following.
In Section~\ref{s:prelim} we recall some known results
on $BV$ functions, divergence-measure vector fields and their
weak normal traces. 
In Section~\ref{s:luno}
we focus our attention on the summability of
$\prec{u}$ with respect to the measure $|\Div\A|$
and on some related properties of the truncated functions.
In Sections~\ref{s:genpair}, \ref{s:coarea} and \ref{s:chain} we introduce the generalized pairing
and we prove that 
it inherits from the standard one
all the main
properties and features.
More precisely, $\pair{\A, Du}$ is a Radon measure,
absolutely continuous with respect to $|Du|$,
it satisfies the coarea, the chain rule and the Leibniz formulas,
and it is consistent with the Gauss--Green formula.

The proofs of these results are based on
the analogous properties valid for the standard pairing
(see \cite{CD3}), the fact that the generalized pairing
differs from the standard one only by a term concentrated
on $J_u$ (see \eqref{f:resto}),
and some representation results of the normal traces of $\A$
on $J_u$ (see \cite{AmbCriMan}).

\smallskip
Our main application 
of the above theory is proposed in Section~\ref{s:sc},
where we consider the semicontinuity properties of the functionals
$F_\varphi$ defined in \eqref{f:Fu},
with respect to the strict convergence in $BV$. 
In Theorem~\ref{t:lsc} we prove
the characterizations \eqref{f:lsc0}--\eqref{f:usc0}
of the semicontinuous pairings. 
The proof is based 
on a recent result of Lahti (see \cite{La}), which assures the lower (upper) 
semicontinuity of the lower $u^-$ (upper $u^+$) limit under the strict 
convergence in $BV$,
combined with the one-sided approximation result
in \cite{CDLP},
and a very careful treatment of the jump part of the measure $\Div\A$.
We show by easy examples that no semicontinuity property has to be expected
with respect to the weak${}^*$ convergence in $BV$.

\section{Notation and preliminary results}
\label{s:prelim}

In the following \(\Omega\) will always denote a nonempty open subset of 
\(\R^N\). 
For every $E\subset \Omega$, $\chi_{E}$ denotes its characteristic function. 
We say that $E_h$ converges to $E$ if $\chi_{E_h}$ converges to $\chi_{E}$ in 
$L^1(\Omega)$. 

We denote by
$\LLN$ 
and $\hh$
the Lebesgue measure 
and the $(N-1)$--dimensional 
Hausdorff measure in $\R^N$, respectively.

If \(E\subset\R^N\) is an open set,
the notation $\varphi \nearrow \chi_{E}$ denotes any family $(\varphi_j)$ of 
smooth functions with support in $E$, such that $0\leq \varphi_j \leq 1$, and $\lim_j \varphi_j(x)=1$
for every $x\in E$.

Given an \(\LLN\)-measurable set \(E\subset\R^N\),
For every \(t\in [0,1]\) we denote by \(E^t\) the set
\[
E^t := \left\{x\in\R^N:\
\lim_{\rho\to 0^+} \frac{\LLN(E\cap B_\rho(x))}{\LLN(B_\rho(x))} = t\right\}
\]
of all points where \(E\) has density \(t\).
The sets \(E^0\), \(E^1\), \(\partial^e E := \R^N\setminus (E^0 \cup E^1)\) are 
called 
respectively the \textsl{measure theoretic exterior}, 
the \textsl{measure theoretic interior} and
the \textsl{essential boundary} of \(E\).

\smallskip
Let $u\colon \Omega\to\R$ be a Borel function.
We denote by $\umeno$ and $\upiu$ 
the \textsl{approximate lower limit} and the 
\textsl{approximate upper limit} of $u$,
defined respectively by
\begin{gather*}
u^+(x) := \inf\{t\in\R:\ \{u>t\}\ \text{has density $0$ at $x$}\},
\\
u^-(x) := \sup\{t\in\R:\ \{u>t\}\ \text{has density $1$ at $x$}\}.
\end{gather*}
The function $u$ is \textsl{approximately continuous} at $x\in\Omega$
if $\upiu(x) = \umeno(x)$ and, in this case, we denote by
$\ut(x)$ the common value.

Given \(u\in L^1_{{\rm loc}}(\Omega)\),
$x\in\Omega$ is a \textsl{Lebesgue point} 
of $u$ (with respect to $\LLN$)
if there exists \(z\in\R\) such that
\[
\lim_{r\rightarrow0^{+}}\frac{1}{\LLN\left(  B_r(x)\right)}\int_{B_r\left(  
x\right)
}\left|  u(y)  -z  \right|  \,dy=0.
\]
In this case, $x$ is a point of approximate
continuity, and $z = \ut(x)$
(see \cite[Proposition~1.163]{FonLeoBook}). 
We denote by \(S_u\subset\Omega\) the set of points where this property does not hold.

We say that \(x\in\Omega\) is an {\sl approximate jump point} of \(u\) if
there exist \(a,b\in\R\) and a unit vector \(\nu\in\R^n\) such that \(a\neq b\)
and
\begin{equation}\label{f:disc}
\begin{gathered}
\lim_{r \to 0^+} \frac{1}{\LLN(B_r^i(x))}
\int_{B_r^i(x)} |u(y) - a|\, dy = 0,
\\
\lim_{r \to 0^+} \frac{1}{\LLN(B_r^e(x))}
\int_{B_r^e(x)} |u(y) - b|\, dy = 0,
\end{gathered}
\end{equation}
where \(B_r^i(x) := \{y\in B_r(x):\ (y-x)\cdot \nu > 0\}\), and 
\(B_r^e(x) := \{y\in B_r(x):\ (y-x)\cdot \nu < 0\}\).
The triplet \((a,b,\nu)\), uniquely determined by \eqref{f:disc} 
up to a permutation
of \((a,b)\) and a simultaneous change of sign of \(\nu\),
is denoted by \((\uint(x), \uext(x), \nu_u(x))\).
The set of approximate jump points of \(u\) will be denoted by \(J_u\).

\smallskip
A $\Haus{N-1}$-measurable set $E\subset\R^N$
is \textsl{countably $\Haus{N-1}$-rectifiable}
if there exist countably many $C^1$ graphs $(\Sigma_i)_{i\in\N}$
such that
$\Haus{N-1}\left(E \setminus \bigcup_i \Sigma_i\right) = 0$.

\subsection{Measures}

The space of all Radon measures on $\Omega$ will be denoted by $\radon$.
 
Given $\mu \in \radon$, its {\sl total variation} $|\mu|$ is 
the nonnegative Radon measure defined by
\[
|\mu|(E) := \sup\left\{ \sum_{h=0}^\infty |\mu(E_h)| \colon \ E_h\ 
\text{$\mu$-measurable sets, pairwise disjoint},\ E=\bigcup_{h=0}^\infty E_h \right\},
\]
for every $\mu$-measurable set $E$
and its {\sl positive} and {\sl negative parts} are defined, respectively, by
\[
\mu^+ := \frac{|\mu|+\mu}{2}, \qquad \mu^- := \frac{|\mu|-\mu}{2}\,.
\]
If $\mu_1, \mu_2\in\radon$, then $\max\{\mu_1, \mu_2\}$
(resp.\ $\min\{\mu_1, \mu_2\}$) is the measure that assigns to every 
Borel set $E\subset\Omega$, the supremum (resp.\ infimum) of
$\mu_1(E_1) + \mu_2(E_2)$ among all pairwise disjoint Borel sets $E_1, E_2$
such that $E_1 \cup E_2 = E$.

Given $\mu\in\radon$ and a $\mu$-measurable set $E$, the \text{restriction} 
$\mu\res E$
is the Radon measure defined by
\[
\mu\res E(B)=\mu(E\cap B), \qquad \forall\ B\ \text{$\mu$-measurable},\ B\subset\Omega.
\]
We recall the following property (see \cite{AFP}, Proposition~2.56
and formula (2.41)):
\begin{equation}\label{f:opn}
E\subset \Omega, \ |\mu|(E)=0\quad  \Longrightarrow\quad  |\mu|(B_r(x))=o(r^{N-1})\ 
\text{for 
$\hh$--a.e.}\ x\in E. 
\end{equation}

Given a nonnegative Borel 
measure $\nu$, we say that $\mu\in\radon$ is 
{\sl absolutely continuous} with respect to $\nu$ (and we write
$\mu \ll \nu$), if $|\mu|(B)=0$ for every set $B$ such that $\nu(B)=0$.

We say that two positive measures $\nu_1$, $\nu_2\in\radon$ are {\sl mutually 
singular}
(and we write $\nu_1 \perp \nu_2$) 
if there exists a Borel set $E$ such that $|\nu_1|(E)=0$ and 
$|\nu_2|(\Omega\setminus E) = 0$.

By the Radon--Nikod\'ym theorem, given a nonnegative Radon measure $\nu$, every 
$\mu\in\radon$ can be uniquely decomposed as $\mu=\mu_1+\mu_2$ with
$\mu_1 \ll \nu$ and $\mu_2 \perp \nu$, and there exists a unique function
(called the density of $\mu$ with respect to $\nu$) 
$\psi_\nu\in L^1(\Omega, \nu)$ such that $\mu_1=\psi_\nu \nu$.
In particular, since $\mu \ll |\mu|$, then there exists
$\psi\in L^1(\Omega, |\mu|)$, with $|\psi|=1$ $|\mu|$--a.e. in $\Omega$, and such 
that  $\mu = \psi |\mu|$. This is usually called the {\sl polar decomposition}
of $\mu$. 

The following lemma shows the relation between the densities of
$\mu$ and $|\mu|$,
where $\mu$ is a Radon measure
absolutely continuous with respect to $\Haus{N-1}$.

\begin{lemma}
\label{r:density}
Let $\mu\ll\Haus{N-1}$ be a Radon measure in $\Omega$,
and let $\mu = \psi |\mu|$ be its polar decomposition.
Then there exists a Borel set $Z\subset\Omega$,
with $|\mu|(Z) = 0$, such that 
every $x\in\Omega\setminus Z$ is a Lebesgue point of $\psi$ with
respect to $|\mu|$, 
and
\begin{equation}\label{f:equald}
\exists\ \lim_{r\searrow 0} \frac{|\mu|(B_r(x))}{r^{N-1}}=L\in\R
\quad\Longleftrightarrow\quad
\exists\ \lim_{r\searrow 0} \frac{\mu(B_r(x))}{r^{N-1}} = \psi(x)\, L.
\end{equation}
\end{lemma}
\begin{proof} 
Let $A\subset\Omega$ be the set of Lebesgue points of $\psi$ with
respect to $|\mu|$.
By \cite[Corollary~2.23]{AFP}, we have that $|\mu|(\Omega\setminus A) = 0$.
Since $|\psi| = 1$ $|\mu|$-a.e., it is not restrictive to assume that
\[
|\psi(x)| = 1,
\quad
\lim_{r\searrow 0} \frac{1}{|\mu|(B_r(x))}
\int_{B_r(x)} |\psi(y) - \psi(x)|\, d|\mu| = 0,
\qquad \forall x\in A.
\]
Moreover, from \cite[Theorem~2.56 and (2.40)]{AFP}, the set
\[
Z_1 := \left\{
x\in\Omega:\
\limsup_{r\searrow 0} \frac{|\mu|(B_r(x))}{r^{N-1}} = +\infty
\right\}
\]
has zero $\hh$-measure, hence also $|\mu|(Z_1) = 0$.

If we set $Z := (\Omega\setminus A)\cup Z_1$, then
$|\mu|(Z) = 0$ and \eqref{f:equald} holds in $\Omega\setminus Z$.
Specifically, given $x\in \Omega\setminus Z$,
$B_r(x) \subset\Omega$ and $\varphi\in C_c(\Omega)$ with
support in $B_r(x)$,
since $|\psi(x)| = 1$,
we have that
$|1 - \psi(y)\psi(x)| = |\psi(y) - \psi(x)|$,
and hence
\[
\begin{split}
\left|
\int_\Omega \varphi\, d|\mu| - \psi(x) \int_\Omega\varphi\, d\mu
\right|
& =
\left|
\int_{B_r(x)} \varphi(y) [ 1 - \psi(y)\psi(x)]\, d|\mu|(y)
\right|
\\ & \leq 
\|\varphi\|_\infty 
|\mu|(B_r(x))\,
\mean{B_r(x)} |\psi(y) - \psi(x)|\, d|\mu|(y)\,.
\end{split}
\]
Taking $\varphi \nearrow \chi_{B_r(x)}$ and dividing
by $r^{N-1}$ we finally get
\[
\left|
\frac{|\mu|(B_r(x))}{r^{N-1}} -
\psi(x)\,\frac{\mu(B_r(x))}{r^{N-1}}
\right|
\leq
\frac{|\mu|(B_r(x))}{r^{N-1}}\,
\mean{B_r(x)} |\psi(y) - \psi(x)|\, d|\mu|(y)\,,
\]
hence \eqref{f:equald} follows because $x\not\in Z_1$ and
$x$ is a Lebesgue point of $\psi$.
\end{proof}

Given $\mu\in\radon$, we denote by $\mu = \mu^a + \mu^s$ its
Lebesgue decomposition in the absolutely continuous part
$\mu^a \ll \LLN$ and the singular part $\mu^s\perp\LLN$.
We recall a relevant decomposition result for $\mu^s$ (see \cite{ADM}, 
Proposition 5). 

\begin{proposition}\label{p:decmu}
If $\mu\in\radon$ is such that $\mu^s \ll \hh$, then $\mu^s$ can be uniquely 
decomposed as the sum $\mu^j+\mu^c$, where $\mu^j$, $\mu^c\in \radon$ 
are two mutually singular measures
having the following properties:
\begin{itemize}
\item[(i)] $\mu^c(B)=0$ for every $B$ such that $\hh(B)<+\infty$;
\item[(ii)] the set
\[
\Theta_\mu:=\left\{ x\in\Omega \colon \limsup_{r \to 0+}
\frac{|\mu|(B_r(x))}{r^{N-1}}>0\right\}
\]
is a Borel set, $\sigma$--finite with respect to $\hh$; 
\item[(iii)] there exists $f\in L^1(\Theta_\mu,\hh\res \Theta_\mu)$ such that 
$\mu^j=f\, \hh\res \Theta_\mu$.
\end{itemize}
The measures $\mu^j$, $\mu^c$ are called {\sl jump part} and {\sl Cantor part} 
of the measure $\mu$, while $\Theta_\mu$ is called {\sl jump set} of $\mu$.
\end{proposition}

\subsection{Functions of bounded variation}
\label{ss:BV}

We say that \(u\in L^1(\Omega)\) is a \textsl{function of bounded variation} in 
\(\Omega\)
if the distributional derivative \(Du\) of \(u\) is a finite Radon measure in 
\(\Omega\).
The vector space of all functions of bounded variation in \(\Omega\)
will be denoted by \(BV(\Omega)\).
Moreover, we will denote by \(BV_{{\rm loc}}(\Omega)\) the set of functions
\(u\in L^1_{{\rm loc}}(\Omega)\) that belongs to 
\(BV(A)\) for every open set \(A\Subset\Omega\)
(i.e., the closure \(\overline{A}\) of \(A\) is a compact
subset of \(\Omega\)).

If \(u\in BV(\Omega)\), then \(Du\) can be decomposed as
the sum of the absolutely continuous and the singular part with respect
to the Lebesgue measure, i.e.\
\[
Du = D^a u + D^s u,
\qquad D^a u = \nabla u \, \LLN,
\]
where \(\nabla u\) is the \textsl{approximate gradient} of \(u\),
defined \(\LLN\)-a.e.\ in \(\Omega\)
(see \cite[Section~3.9]{AFP}).
The jump set $J_u$ has the following properties:
it is countably $\H^{N-1}$--rectifiable
and 
$\H^{N-1}(S_u  \setminus J_u) = 0$
(see \cite[Definition~2.57 and Theorem~3.78]{AFP});
it is contained in the set $\Theta_{Du}$ defined in
Proposition~\ref{p:decmu}(ii) with $\mu = Du$,
and $\Haus{N-1}(\Theta_{Du}\setminus J_u) = 0$
(see \cite[Proposition~3.92(b)]{AFP}).
By Proposition~\ref{p:decmu},
the singular part \(D^s u\) can be further decomposed
as the sum of its Cantor and jump part, i.e.
$D^s u = D^c u + D^j u$,
$D^c u := D^s u \res (\Omega\setminus S_u)$,
and
\[
D^j u := D^s u \res J_u = (\prec[i]{u} - \prec[e]{u}) \, \nu_u\, \Haus{N-1}\res J_u.
\]
We denote by \(D^d u := D^a u + D^c u\) the diffuse part of the measure 
\(Du\).

At every point $x\in J_u$ we have that
$-\infty < \umeno(x) < \upiu(x) < +\infty$ and
\[
\umeno(x) =\min\{\uint(x),\uext(x)\}, \qquad 
\upiu(x) =\max\{\uint(x),\uext(x)\},
\qquad x\in J_u.
\]
Moreover, we can always choose an orientation on $J_u$
such that $\uint = \upiu$ on $J_u$
(see \cite[\S 4.1.4, Theorem~2]{GMS1}).
In the following we shall always extend the functions $\uint,\uext$ to
\(\Omega\setminus(S_u\setminus J_u)\) by setting
\[
\uint=\uext=\widetilde{u}\quad \text{in}\ \Omega\setminus S_u.
\]

Given a Borel function $\lambda \colon \Omega \to [0,1]$, the {\sl 
$\lambda$--representative} of $u\in BV_{\rm loc}(\Omega)$ is defined by
\begin{equation}\label{f:pr}
\prec{u}(x):=
\begin{cases}
\tilde{u}(x), & x\in \Omega \setminus S_u, \\
(1-\lambda(x))\umeno(x)+\lambda(x)\upiu(x), & x\in J_u.
\end{cases}
\end{equation} 
When $\lambda(x) = 1/2$ for every $x\in\Omega$,
the $\lambda$--representative coincides with the {\sl precise representative} $u^* := (\upiu + \umeno)/2$ of $u$.

\smallskip
Let \(E\) be an \(\LLN\)-measurable subset of \(\R^N\).
For every open set \(\Omega\subset\R^N\) the \textsl{perimeter} \(P(E, \Omega)\)
is defined by
\[
P(E, \Omega) := \sup\left\{
\int_E \Div \varphi\, dx:\ \varphi\in C^1_c(\Omega, \R^N),\ 
\|\varphi\|_\infty\leq 1
\right\}.
\]
We say that \(E\) is of \textsl{finite perimeter} in \(\Omega\) if \(P(E, \Omega) < 
+\infty\).

Denoting by \(\chi_E\) the characteristic function of \(E\),
if \(E\) is a set of finite perimeter in \(\Omega\), then
\(D\chi_E\) is a finite Radon measure in \(\Omega\) and
\(P(E,\Omega) = |D\chi_E|(\Omega)\).

If \(\Omega\subset\R^N\) is the largest open set such that \(E\)
is locally of finite perimeter in \(\Omega\),
we call \textsl{reduced boundary} \(\partial^* E\) of \(E\) the set of all points
\(x\in \Omega\) in the support of \(|D\chi_E|\) such that the limit
\[
\nuint_E(x) := \lim_{\rho\to 0^+} 
\frac{D\chi_E(B_\rho(x))}{|D\chi_E|(B_\rho(x))}
\]
exists in \(\R^N\) and satisfies \(|\nuint_E(x)| = 1\).
The function \(\nuint_E\colon\partial^* E\to S^{N-1}\) is called
the \textsl{measure theoretic unit interior normal} to \(E\).

A fundamental result of De Giorgi (see \cite[Theorem~3.59]{AFP}) states that
\(\partial^* E\) is countably \((N-1)\)-rectifiable
and \(|D\chi_E| = \hh\res \partial^* E\).
If \(E\) has finite perimeter in \(\Omega\), Federer's structure theorem
states that
\(\partial^* E\cap\Omega \subset E^{1/2} \subset \partial^e E\)
and \(\H^{N-1}(\Omega\setminus(E^0\cup \partial^e E \cup E^1)) = 0\)
(see \cite[Theorem~3.61]{AFP}).

\subsection{Divergence--measure fields }
\label{ss:div}

We will denote by \(\DM(\Omega)\) the space of all
vector fields 
\(\A\in L^\infty(\Omega, \R^N)\)
whose divergence in the sense of distributions is a finite Radon measure in 
\(\Omega\), acting as
\[
\int_\Omega \varphi\, d\Div\A = -\int_{\Omega} \A\cdot\nabla\varphi\, dx \qquad
\forall \varphi\in\cinftio.
\]
Similarly, \(\DMlocloc[\Omega]\) will denote the space of
all vector fields \(\A\in L^\infty_{{\rm loc}}(\Omega, \R^N)\)
whose divergence in the sense of distributions is a Radon measure in 
\(\Omega\). 

The basic properties of these vector fields are collected in the following 
proposition.

\begin{proposition}\label{p:basicVF}
Let \(\A\) be a vector field belonging to \(\DMloc[\Omega]\), and let 
$\jump{\A}$ be the jump set of the measure $\mu = |\Div\A|$,
defined in Proposition~\ref{p:decmu}(ii).
Then the following hold.
\begin{itemize}
\item[(i)] \(|\Div\A| \ll \hh\);
\item[(ii)] $\jump{\A}$ is a Borel set, \(\sigma\)-finite with respect to 
\(\hh\);
\item[(iii)]\(
\Div\A = \Div^a\A + \Div^c\A + \Div^j\A,
\)
where \(\Div^a\A\) is absolutely continuous with respect to \(\LLN\),
\(\Div^c\A (B) = 0\) for every set \(B\) with \(\hh(B) < +\infty\),
and there exists $f\in L^1(\jump{\A}, \Haus{N-1}\res \jump{\A})$
such that
\(
\Div^j\A = f\, \hh\res\jump{\A}
\).
\end{itemize}
\end{proposition}
\begin{proof}
The main property (i) is proved in \cite[Proposition 3.1]{ChenFrid}.
The decomposition then follows from Proposition \ref{p:decmu}.
\end{proof}

\subsection{Weak normal traces}
\label{distrtraces}
In what follows, we
will deal with the traces of the normal component of a vector field \(\A\in 
\DMloc\) on a countably \(\H^{N-1}\)--rectifiable set
\(\Sigma\subset\Omega\).
In order to fix the notation, we briefly recall the construction given in 
\cite{AmbCriMan} (see Propositions~3.2, 3.4 and Definition~3.3).

Given a domain \(\Omega'\Subset\Omega\) of class \(C^1\), 
the trace of the normal component of \(\A\) on \(\partial\Omega'\) 
is the distribution defined by
\begin{equation}\label{f:disttr}
\pscal{\Trace[]{\A}{\partial\Omega'}}{\varphi}
:= \int_{\Omega'} \A\cdot \nabla\varphi\, dx + \int_{\Omega'} \varphi\, d\Div\A,
\qquad
\forall\varphi\in C^\infty_c(\Omega).
\end{equation}
It turns out that this distribution is induced by an \(L^\infty\) function on 
\(\partial\Omega'\),
still denoted by \(\Trace[]{\A}{\partial\Omega'}\), and
\begin{equation}\label{f:infestrace}
\|\Trace[]{\A}{\partial\Omega'}\|_{L^\infty(\partial\Omega', \Haus{N-1}\res \partial\Omega')}
\leq \|\A\|_{L^\infty(\Omega')}.
\end{equation}

Given a countably $\hh$--rectifiable set \(\Sigma\), there exist a covering
\((\Sigma_i)_{i\in\N}\) of \(\Sigma\) and  
Borel sets $N_i\subseteq \Sigma_i$ with the following properties:
\begin{itemize}
\item[(R1)]  \(\Sigma_i\) is an \textsl{oriented} \(C^1\) hypersurface,
with (classical) normal vector field \(\nu_{\Sigma_i}\);
\item[(R2)] \(N_i\subseteq \Sigma_i\) are pairwise disjoint Borel sets
such that \(\hh(\Sigma\setminus \bigcup_i N_i) = 0\); 
\item[(R3)] for every $i\in\N$, there exist two open bounded sets \(\Omega_i, \Omega'_i\) with 
\(C^1\) boundary
and exterior normal vectors \(\nu_{\Omega_i}\) and \(\nu_{\Omega_i'}\) 
respectively,
such that
\(N_i\subseteq \partial\Omega_i \cap \partial\Omega'_i\),
and
\[
\nu_{\Sigma_i}(x) = \nu_{\Omega_i}(x) = -\nu_{\Omega'_i}(x)
\qquad \forall x\in N_i.
\]
\end{itemize} 

We can fix an orientation on \(\Sigma\), given by
\[
\nu_{\Sigma}(x) := \nu_{\Sigma_i}(x), \qquad \hh-\text{a.e. on}\  N_i
\]
and the normal traces of \(\A\) on \(\Sigma\) 
are defined by
\[
\Trm{\A}{\Sigma} := \Tr(\A, \partial\Omega_i),
\quad
\Trp{\A}{\Sigma} := -\Tr(\A, \partial\Omega'_i),
\qquad
\hh-\text{a.e.\ on}\ N_i.
\]
By a deep localization property proved in \cite[Proposition 3.2]{AmbCriMan},
these definitions are independent of the choice
of $\Sigma_i$ and $N_i$. 
In what follows, the pair $(\Sigma, \nu_\Sigma)$ (or, simply, $\Sigma)$ will be called and oriented countably $\hh$-rectifiable set.

We remark that, the normal traces belong to
\(L^{\infty}(\Sigma, \H^{N-1}\res\Sigma)\) 
and
\begin{equation}\label{f:trA}
\Div \A \res\Sigma =
\left[\Trp{\A}{\Sigma} - \Trm{\A}{\Sigma}\right]
\, {\mathcal H}^{N-1} \res\Sigma
\end{equation}
(see \cite[Proposition~3.4]{AmbCriMan}).
In particular, by \eqref{f:infestrace},
$|\Div\A|(\Sigma) \leq \|\A\|_\infty \Haus{N-1}(\Sigma)$.

\begin{remark}\label{r:diffnt}
We observe that, if $\Sigma$ is oriented by a normal vector field
$\nu$ and $\Sigma'$ is the same set oriented by $\nu' := -\nu$,
then
\[
\Trm{\A}{\Sigma'} = -\Trp{\A}{\Sigma},
\quad
\Trp{\A}{\Sigma'} := -\Trm{\A}{\Sigma},
\]
so that the difference
$\Trp{\A}{\Sigma} - \Trm{\A}{\Sigma}$
is independent of the choice of the orientation on $\Sigma$.
\end{remark}

The following result is a consequence of \eqref{f:trA} and will be used in the study of
the semicontinuity of the generalized pairing
(see Theorem~\ref{t:lsc}).

\begin{theorem}
\label{t:density}
Let $\A\in\DMloc$, 
let $\Div\A = \psi_{\A} |\Div\A|$ be the polar decomposition
of the measure $\Div\A$,
and let $\Sigma\subset\Omega$ be an oriented
countably $\hh$-rectifiable set.
Then
\begin{gather}
\Trp{\A}{\Sigma}(x) - \Trm{\A}{\Sigma}(x)
=
\lim_{r\searrow 0} \frac{\Div\A\, (B_r(x))}{\omega_{N-1} r^{N-1}}\,,
\qquad
\text{for $\hh$-a.e.}\ x\in\Sigma\,,
\label{f:density}\\
\Trp{\A}{\Sigma}(x) - \Trm{\A}{\Sigma}(x)
=
\psi_{\A}(x)\,\lim_{r\searrow 0} \frac{|\Div\A|(B_r(x))}{\omega_{N-1} r^{N-1}}\,,
\quad
\text{for $\hh$-a.e.}\ x\in\Sigma\,.
\label{f:density2}
\end{gather}
\end{theorem}

\begin{proof}
From \eqref{f:opn} with $\mu := |\Div\A|\res\Sigma$
and $E:= \Omega\setminus\Sigma$,
we have that
\[
\lim_{r\searrow 0} \frac{|\Div\A|\res(\R^N\setminus\Sigma)\, (B_r(x))}{\omega_{N-1} r^{N-1}} = 0\,,
\qquad
\text{for $\Haus{N-1}$-a.e.\ $x\in\Sigma$}.
\]
On the other hand, by \eqref{f:trA}
\[
\lim_{r\searrow 0} \frac{\Div\A\res\Sigma\, (B_r(x))}{\omega_{N-1} r^{N-1}}
= \Trp{\A}{\Sigma}(x) - \Trm{\A}{\Sigma}(x)\,,
\qquad
\text{for $\Haus{N-1}$-a.e.\ $x\in\Sigma$}\,,
\]
and hence \eqref{f:density} holds.

Let us define the sets
\begin{gather*}
\Sigma' :=
\left\{
x\in \Sigma:\ \exists
\lim_{r\searrow 0} \frac{|\Div\A|(B_r(x))}{\omega_{N-1} r^{N-1}} = 0
\right\}\,,\\
\Sigma'' :=
\left\{
x\in \Sigma:\ \exists
\lim_{r\searrow 0} \frac{|\Div\A|(B_r(x))}{\omega_{N-1} r^{N-1}} > 0
\right\}\,.
\end{gather*}
By Proposition~\ref{p:basicVF}(i) and \cite[Theorems~2.22 and~2.83]{AFP}, 
we infer that
$\Haus{N-1}(\Sigma\setminus(\Sigma'\cup\Sigma'')) = 0$.
From  \eqref{f:density} and Lemma~\ref{r:density}
we deduce that the equality in \eqref{f:density2} holds 
for $\Haus{N-1}$-a.e.\  $x\in\Sigma''$.
On the other hand, from \eqref{f:density} we deduce that
$\Trp{\A}{\Sigma}(x) - \Trm{\A}{\Sigma}(x) = 0$ for
$\Haus{N-1}$-a.e.\ $x\in\Sigma'$,
hence \eqref{f:density2} follows.
\end{proof}

For later use, we recall here a result
proved in \cite[Proposition~3.1]{CD3}.

\begin{proposition}\label{p:tracesb}
	Let \(\A\in\DMloc\),  \(u\in\BVLloc\) and let \(\Sigma\subset\Omega\) 
	be an oriented countably \(\hh\)--rectifiable set.
	Then \(u\A\in\DMloc\) and the normal traces of \(u\A\) on \(\Sigma\) are 
	given by
\[
	\begin{split}
		\Trm{u\A}{\Sigma} = 
		& \begin{cases}
			\uext \Trm{\A}{\Sigma},
			& \hh-\text{a.e.\ in}\ J_u\cap\Sigma,\\
			\ut\, \Trm{\A}{\Sigma},
			& \hh-\text{a.e.\ in}\ \Sigma\setminus J_u. 	
		\end{cases} \\ 
		\Trp{u\A}{\Sigma} =
		&\begin{cases}
					\uint \Trp{\A}{\Sigma},
					& \hh-\text{a.e.\ in}\ J_u\cap\Sigma,\\
					\ut\, \Trp{\A}{\Sigma},
					& \hh-\text{a.e.\ in}\ \Sigma\setminus J_u. 	
				\end{cases} 
	\end{split}
\]
\end{proposition}

\section{Some remarks on $L^1(\Omega, |\Div\A|)$}
\label{s:luno}

In this section we analyze the properties of the functional spaces
needed to define the pairing \(\pair{\A, Du}\) 
introduced in \eqref{f:genp}.

\begin{definition}\label{d:BVA}
Given \(\A\in\DM(\Omega)\), let
us define the spaces:
\begin{gather*}
\BVA := \left\{
u\in BV(\Omega):\ u^* \in L^1(\Omega, |\Div\A|)
\right\}\,,
\\
\BVAlocloc := \left\{
u\in BV_{\rm loc}(\Omega):\ u^* \in L^1_{\rm loc}(\Omega, |\Div\A|)
\right\}\,.
\end{gather*}
\end{definition}
Notice that $|\Div\A| \ll \Haus{N-1}$ and $u^*$ is defined
$\Haus{N-1}$-a.e.\ in $\Omega$,
hence the definitions are well-posed.

The following lemma shows that if  $u\in\BVAloc$ then any representative
$\prec[\lambda]{u}$ of $u$ defined in \eqref{f:pr} (in particular $\upiu$, $\umeno$) is summable with 
respect to the measure $|\Div \A|$,
hence the definitions of the spaces
$\BVA$ and $\BVAloc$ are independent of the choice of the pointwise
representative. 

\begin{lemma}\label{l:l1}
Let $\A\in\DMloc$
and let $u\in BV_{{\rm loc}}(\Omega)$.
Given two Borel selections $\lambda, \mu \colon\Omega\to [0,1]$,
then it holds:
\begin{itemize}
\item[(i)]
\(
\prec{u}\in L^1_{\rm{loc}}(\Omega,|\Div \A|)
\)
if and only if
\(
\prec[\mu]{u}\in L^1_{\rm{loc}}(\Omega,|\Div \A|)
\);
\item[(ii)]
for every countably
$\Haus{N-1}$-rectifiable set
$\Sigma\subset\Omega$,
\(
\prec{u}\in L^1_{\rm{loc}}(\Sigma,\Haus{N-1}\res\Sigma)
\)
if and only if
\(\prec[\mu]{u}\in L^1_{\rm{loc}}(\Sigma,\Haus{N-1}\res\Sigma)\).
\end{itemize}
\end{lemma}

\begin{proof}
We prove only (i), being the proof of (ii) entirely similar.
By the representation \eqref{f:trA}
of $\Div\A\res J_u$
and the estimate \eqref{f:infestrace},
for every compact set $K\Subset \Omega$ we have 
\[
\begin{split}
\int_{J_u \cap K} (\upiu-\umeno)\, d |\Div\A| & =
\int_{J_u \cap K} (\upiu-\umeno) |\Trp{\A}{J_u}-\Trm{\A}{J_u}|\, d \hh \\
& \leq 
2 \|\A\|_{L^\infty(K)}
|D^j u|(K).
\end{split}
\]
Recalling that $\upiu - \umeno = 0$ in $\Omega\setminus S_u$,
i.e.\ $\hh$-a.e.\ in $\Omega\setminus J_u$,
it follows that $\upiu-\umeno \in L^1_{\rm{loc}}(\Omega,|\Div \A|)$.
The result now follows by observing that
$
\prec{u} = \prec[\mu]{u} + (\lambda-\mu) (u^+-u^-).
$
\end{proof}

We underline that, 
for every $\A\in\DM(\Omega)$
and every $u\in BV(\Omega)$, it holds
\[
\int_{J_u} |\Trace[i,e]{\A}{J_u}|\, (\prec[+]{u} - \prec[-]{u})\, d\Haus{N-1}
\leq {\|\A\|}_\infty |D^j u| (\Omega) < +\infty.
\]
Nevertheless, in general the functions
$|\Trace[i,e]{\A}{J_u}|\,  u^{\pm}$ are not summable
with respect to $\Haus{N-1}\res J_u$,
even under the additional assumption $u\in \BVA[\Omega]$,
as it is shown in the following example.

\begin{example}
Let $\Omega = B_1(0) \subset \R^2$.
Let us show that there exist a vector field $\A\in\DM(\Omega)$
and a function $u\in\BVA$ such that
\[
\int_{J_u} |\Trace[i,e]{\A}{J_u}|\,  u^{\pm} \, d\Haus{1} = +\infty.
\]
Let $1 = r_0 > r_1 > \cdots > r_n > \cdots $ be a decreasing sequence
converging to $0$, such that 
\[
\sum_j r_j < +\infty,
\qquad
\sum_j j\, r_j = +\infty,
\]
and let
$u\colon \Omega\to \R$ be defined by
$u(x) = j$, if $r_j \leq |x| < r_{j-1}$, $j\in\N$.
Since
\[
\int_\Omega u\, dx = \pi \sum_{j=1}^{\infty}r_j^2 < \infty,
\quad
Du = \sum_{j=1}^\infty \Haus{1}\res\partial B_{r_j}(0)\,,
\quad
|Du|(\Omega) = 2\pi \sum_{j=1}^\infty r_j < \infty,
\]
then $u\in BV(\Omega)$.
We choose on the jump set $J_u = \bigcup_{j=1}^\infty \partial B_{r_j}(0)$
the orientation such that $u^i = u^+ = j+1$ and $u^e = u^- = j$
on $\partial B_{r_j}(0)$.

Let $(a_j)\subset\R$ be a bounded sequence, and let
\[
\A(x) := a(|x|)\, \frac{x}{|x|}\,,
\qquad
\text{with}\quad
a(\rho) :=  \sum_{j=1}^{\infty} a_j\, \chi_{[r_j, r_{j-1})}(\rho),
\
\rho \in (0,1).
\]
We have that $\A\in L^\infty(\Omega, \R^2)$, 
$\Trace[i]{\A}{J_u} = a_{j+1}$,
$\Trace[e]{\A}{J_u} = a_j$
on $\partial B_{r_j}$, and
\begin{gather*}
\Div\A = \frac{a(|x|)}{|x|}\,\Leb{2} + \sum_{j=1}^{\infty} (a_{j+1}-a_j)\Haus{1}\res\partial B_{r_j}\,,
\\
|\Div\A|(\Omega) \leq {\|\A\|}_\infty \int_\Omega \frac{1}{|x|}\, dx
+ 2\pi \sum_{j=1}^{\infty} |a_{j+1}-a_j|\, r_j < +\infty,
\end{gather*}
so that $\A\in\DM(\Omega)$.
On the other hand, if we choose a sequence $(a_j)$
such that $|a_j| \geq c > 0$ for every $j\in\N$, we have that
\[
\int_{J_u} u^- \, |\Trace[i,e]{\A}{J_u}|\, d\Haus{1}
\geq
2\pi c \sum_{j=1}^{\infty} j\, r_j = +\infty.
\]
We conclude this example observing that, with the choice
$a_j = (-1)^j$, we have also that
\[
\int_{J_u} u^\pm \, |\Trace[i]{\A}{J_u} - \Trace[e]{\A}{J_u}|\, d\Haus{1}
= +\infty.
\]
\end{example}

\medskip
We collect here the main features of the truncation operator
that will be useful to generalize
to $u\in\BVA[\Omega]$ properties valid in 
$\BVL[\Omega]$.

\begin{proposition}[Properties of the truncated functions]
\label{p:trunc}
For every $k>0$, let
\begin{equation}\label{gtrun}
T_k(s) := \max\{\min\{s, k\}, -k\}
\,,
\qquad s\in\R.
\end{equation}
Let $u\in BV(\Omega)$
and let $\lambda\colon\Omega\to[0,1]$ be a Borel function.
Then the following hold.
\begin{itemize}
\item[(i)]
$T_k(u^\pm) = [T_k(u)]^\pm \to u^\pm$,
$\prec{[T_k(u)]}\to u^\lambda$,
$\Haus{N-1}$-a.e.\ in $\Omega$;

\item[(ii)]
$|D T_k(u)| \leq |Du|$ in the sense of measures, for every $k>0$;

\item[(iii)]
$|[T_k(u)]^\pm| \leq |u^\pm|$ for every $k>0$, hence
\[
|T_k(\prec{u})| \leq (1-\lambda) |u^-| + \lambda\, |u^+|
\qquad
\forall k>0;
\]

\item[(iv)] if $u\in\BVA[\Omega]$, then
$T_k(\prec{u}) \to \prec{u}$ in $L^1(\Omega, |\Div\A|)$.
\end{itemize}
\end{proposition}

\begin{proof}
The proof of (i) can be found in \cite[Theorem~4.34(a)]{AFP}.

The inequality in (ii) is a consequence of the fact that $T_k$ is a
$1$-Lipschitz function
(see the first part of the proof of Theorem~3.96 in \cite{AFP}).

The inequalities in (iii) follow from $|T_k(s)| \leq |s|$
and the equalities in (i), 
whereas (iv) follows from (iii), Lemma~\ref{l:l1}, 
and Lebesgue's Dominated Convergence Theorem.
\end{proof}

\section{Definition and basic properties of pairings
}
\label{s:genpair}

\begin{definition}[Generalized pairing]\label{d:pair}
Given a vector field $\A\in \DMloc$ and a Borel function 
$\lambda \colon \Omega \to [0,1]$, 
for every $u\in\BVAloc$
the {\sl $\lambda$--pairing} between $\A$ 
and $Du$
is the distribution
$\pair{\A,Du}\colon C^\infty_c(\Omega) \to \R$ acting as
\begin{equation}\label{f:gpairing}
\pscal{\pair{\A,Du}}{\varphi}:= -\intA{\prec{u}\,\varphi}- 
\intom{u\A\cdot\nabla\varphi}\,,
\qquad
\varphi\in C^\infty_c(\Omega)\,.
\end{equation}
\end{definition}

\begin{remark}
The standard pairing
\[
\pscal{\spair{\A, Du}}{\varphi} :=
-\int_\Omega u^*\varphi\, d \Div \A - \int_\Omega u \, \A\cdot \nabla\varphi\, 
dx, 
\]
introduced in \cite{Anz}, and deeply studied in recent years (see e.g.\
\cite{ChenFrid}, \cite{CD3} and the references therein),
is the $\lambda$-pairing corresponding to the
constant selection
$\lambda(x)=\frac{1}{2}$ for every $x\in \Omega$.
\end{remark}

\begin{remark}
The definition of generalized pairing and the properties proved in the
rest of the paper can be extended straightforwardly to
vector fields $\A\in \DMlocloc$ and functions $u\in\BVAlocloc$.
\end{remark}

Clearly, the change of pointwise values of $u$ may just affect the behavior of
the pairing on the jump set $J_u$ of $u$. More precisely, the following basic
properties hold.

\begin{proposition}\label{p:pairing0}
Let $\A\in \DMloc$, $u\in\BVAloc$,  and $\lambda \colon \Omega \to [0,1]$ be a 
Borel function. Then $\pair{\A,Du}$ is a Radon measure in $\Omega$, and 
the equations
\begin{equation}\label{f:servi}
\Div(u\A)=\prec{u}\Div\A+\pair{\A,Du}\,,
\end{equation}
\begin{equation}\label{f:resto}
\pair{\A,Du}=\spair{\A,Du}+\left(\frac{1}{2}-\lambda\right)(\upiu-\umeno) 
\Div\A\res{J_u}\,,
\end{equation}
hold in the sense of measures in $\Omega$.
Moreover, $\pair{\A,Du}$ is absolutely continuous
with respect to $|Du|$, and  
\begin{equation}\label{f:astim}
|\pair{\A,Du}| \leq 
\|\A\|_\infty |Du|.
\end{equation}  
In what follows we will write
\begin{equation}\label{f:polar}
\pair{\A,Du}=\polar(\A,Du,x)|Du|,
\end{equation}
where $\polar(\A,Du,\cdot)$ denotes the Radon--Nikod\'ym derivative of
$\pair{\A,Du}$ with respect to $|Du|$.
\end{proposition}

\begin{proof}
Assume, in addition, that $u\in\BVLloc$.
In this case the fact that $\pair{\A,Du}$ is a Radon measure, and the 
validity of \eqref{f:servi} are 
straightforward 
consequences of the fact that the distribution
\[
\pscal{\Div(u\A)}{\varphi}= \int_\Omega u \, \A\cdot \nabla\varphi\, 
dx, \qquad \varphi\in\cinfc
\] 
is a Radon measure in $\Omega$ (see 
\cite{ChenFrid}). 
Moreover, we have that
\[
\begin{split}
\pair{\A,Du} & =-\ust 
\Div\A+\Div(u\A)+\left(\frac{1}{2}-\lambda\right)(\upiu-\umeno) 
\Div\A\res{J_u}  
\\
& =\spair{\A,Du}+\left(\frac{1}{2}-\lambda\right)(\upiu-\umeno) 
\Div\A\res{J_u}\,.
\end{split}
\]
From \cite[Proposition~3.5]{SchSch} (see in particular formula (3.9) there),
we have that
\[
\left|\pair[0]{\A,Du}\right|\,, \
\left|\pair[1]{\A,Du}\right|
\leq \|\A\|_\infty |Du|.
\]
Since $\pair{\A,Du} = (1-\lambda)\pair[0]{\A,Du} + \lambda\pair[1]{\A,Du}$,
\eqref{f:astim} follows.

\medskip
Consider now the general case $u\in\BVA[\Omega]$.
Let $u_k := T_k(u)$ be the sequence of truncated functions.
By Proposition~\ref{p:trunc}(i) and (iv), we have that
$\prec{(u_k)} \to u^\lambda$
$\Haus{N-1}$-a.e.\ in $\Omega$
and in $L^1(\Omega, |\Div\A|)$.
Hence, we can pass to the limit in
\[
\pscal{\pair{\A,Du_k}}{\varphi} = -\intA{\prec{(u_k)}\,\varphi}- 
\intom{u_k\A\cdot\nabla\varphi}
\]
and obtain 
that $\pair{\A,Du}$ is the weak${}^*$ limit
of $\pair{\A,Du_k}$ in the sense of measures,
so that \eqref{f:servi} follows.
Since, by the estimate \eqref{f:astim} and Proposition~\ref{p:trunc}(ii), 
we have that
\[
|\pair{\A, Du_k}|(\Omega) \leq \|\A\|_\infty |Du_k|(\Omega)
\leq \|\A\|_\infty |Du|(\Omega),
\qquad
\forall k\in\N,
\]
we conclude that  
\eqref{f:servi}, \eqref{f:resto} and \eqref{f:astim}
hold in the sense of measures.
\end{proof}

\begin{remark}
\label{r:trunc}
In the last part of the proof of Proposition~\ref{p:pairing0}
we have shown that,
for every $\A\in\DMloc$ and every $u\in \BVAloc$,
the pairing $\pair{\A,Du}$
is the weak${}^*$ limit, in the sense of measures,
of the sequence $\pair{\A,DT_k(u)}$.
\end{remark}

\begin{remark}\label{r:convcomb} 
Since $(u+v)^+ \leq u^+ + v^+$ and
$(u+v)^- \geq u^- + v^-$,
with possibly strict inequalities,
the map $u\mapsto \pair{\A,Du}$ is not linear, in general.
On the other hand, the map $u\mapsto u^*$ is linear,
hence the standard pairing is linear
with respect to $u$.
More precisely, the $\lambda$-pairing is linear
if and only if 
$\pair{\A,Du} = \spair{\A,Du}$ for every $u\in\BVL$.
Indeed, for every $u\in\BVL$ we have that
\[
\begin{split}
\pair{\A, Du} + \pair{\A, D(-u)}
& = \left(\frac{1}{2}-\lambda\right)
[\upiu-\umeno +(-u)^+ - (-u)^-] \Div\A \res J_u
\\ & =
2 \left(\frac{1}{2}-\lambda\right)(\upiu-\umeno) \Div\A \res J_u.
\end{split}
\]
Hence, if there exists $u\in\BVL$ such that
$\pair{\A,Du} \neq \spair{\A,Du}$, 
then the claim follows from \eqref{f:resto}.
\end{remark}

Using \eqref{f:resto}, and the results of Theorem 3.3 in \cite{CD3}, 
we are able to 
compute explicitly the diffuse part $\pair{\A,Du}^d$, the absolutely continuous part 
$\pair{\A,Du}^a$, and the jump part $\pair{\A,Du}^j$ of the generalized
pairing.

\begin{proposition}\label{p:pairing}
Let \(\A\in\DMloc\) and \(u\in\BVAloc\).
Then the diffuse, the absolutely continuous and the jump part of the measure 
\(\pair{\A, Du}\) are respectively
\begin{gather*}
\pair{\A, Du}^d = \spair{\A, Du}^d\,, \qquad
\pair{\A, Du}^a = \A \cdot \nabla u\, \LLN, 
\\
\pair{\A, Du}^j = 
\left[(1-\lambda)\Trp{\A}{J_u} +\lambda \Trm{\A}{J_u}
\right]\, (\upiu-\umeno) \, 
\hh 
\res J_u\,,
\end{gather*}
where 
$\Trp{\A}{J_u}$ and $\Trm{\A}{J_u}$ 
are the normal traces corresponding to
the orientation of $J_u$ such that $\upiu=\uint$.
\end{proposition} 

\begin{proof}
By \eqref{f:resto}, \(\pair{\A, Du}\) and \(\spair{\A, Du}^d\)
may differ only on $J_u$,
hence \(\pair{\A, Du}^d=\spair{\A, 
Du}^d\). Moreover, by Theorem 3.2 in \cite{ChenFrid}, \(\pair{\A, 
Du}^a=\spair{\A, 
Du}^a=\A \cdot \nabla u\, \LLN\).

Concerning the jump part \(\pair{\A, Du}^j\), by \eqref{f:astim}, we already 
know that it is 
concentrated on $J_u$. 
Denoting by
$\alpha^i := \Trp{\A}{J_u}$ and $\alpha^e := \Trm{\A}{J_u}$, 
by Theorem 3.3 in \cite{CD3} we already know that
\[
\spair{\A, Du}^j = \frac{\alpha^i+\alpha^e}{2}
	\, (\upiu-\umeno) \, \hh \res J_u\,.
\]
Finally, by \eqref{f:resto}
and \eqref{f:trA}, we conclude that 
\[
\begin{split}
\pair{\A, Du}^j & = \spair{\A, Du}^j+
\left(\frac{1}{2}-\lambda\right)(\upiu-\umeno) \Div\A \res J_u \\
& =
\frac{\alpha^i+\alpha^e}{2}
	\, (\upiu-\umeno) \, \hh \res J_u 
+ \left(\frac{1}{2}-\lambda\right)(\upiu-\umeno) (\alpha^i-\alpha^e)
	\, {\mathcal H}^{N-1} \res J_u \\
& = [(1-\lambda)\alpha^i +\lambda \alpha^e]\, (\upiu-\umeno) \, 
\hh 
\res J_u.
\qedhere
\end{split}
\]
\end{proof}

\begin{remark}
[The pairing trivializes on $W^{1,1}$]
From Proposition~\ref{p:pairing}, we have that
\[
\pair{\A, Du} =
\spair{\A, Du}
= \A\cdot \nabla u\, \LLN,
\qquad
\forall u\in W^{1,1}(\Omega)\cap L^\infty(\Omega).
\]
\end{remark}

\begin{remark}
[\(BV\) vector fields]	
If \(\A \in BV(\Omega, \R^N) \cap L^\infty(\Omega, \R^N)\),
then clearly \(\A \in \DMloc\) and
\[
\Trace[i,e]{\A}{J_u} = \A^{i,e}_{J_u}\cdot \nu_u\,, 
\qquad
\text{\(\hh\)-a.e.\ in}\ J_u,
\]
where \(\A^{i,e}_{J_u}\) are the traces of \(\A\) on \(J_u\)
in the sense of $BV$
(see \cite[Theorem~3.77]{AFP}).
Hence, the jump part of \(\pair{\A, Du}\) can be written as
\[
\pair{\A, Du}^j = \left[(1-\lambda)\A^i_{J_u} 
+ \lambda\, \A^e_{J_u}\right]\, \cdot D^j u.
\]
\end{remark}

\begin{remark}
[The pairing trivializes for continuous vector fields]
If $\A\in C(\Omega, \R^N)$, then
by \cite[Theorem~3.3]{CD3} and
\cite[Theorem~3.7]{ComiPayne}
it holds
\[
\pair{\A, Du} =
\A\cdot Du,
\qquad
\forall u\in \BVAloc.
\]
\end{remark}

The following result is an improvement of
Proposition~4.15 in \cite{CD3}, Theorem~1.2 in \cite{ChenFrid}
and Lemma~2.2 in \cite{Anz}. 

\begin{proposition}[Approximation by \(C^\infty\) fields]\label{p:approxinf}
	Let \(\A\in\DM(\Omega)\).
	Then there exists a sequence \((\A_k)_k\) in 
	\(C^\infty(\Omega,\R^N)\cap L^\infty(\Omega, \R^N)\) 
	satisfying the following properties.
	\begin{itemize}
		\item[(i)] 
		\(\A_k - \A\to 0\) in \(L^1(\Omega,\R^N)\), 
		\(\int_\Omega |\Div \A_k|\, dx \to |\Div \A|(\Omega)\),
and \((\A_k)_k\) is uniformly bounded.
		\item[(ii)] 
		\(\Div\A_k \stackrel{*}{\rightharpoonup} \Div\A\)
		in the weak${}^*$ sense of measures in \(\Omega\).
		\item[(iii)]
		For every oriented countably \(\hh\)-rectifiable set \(\Sigma\subset\Omega\) it holds
		\[
		\lim_{k\to +\infty} \pscal{\Trace[i,e]{\A_k}{\Sigma}}{\varphi} =
		\pscal{\Trace[*]{\A}{\Sigma}}{\varphi}
		\qquad \forall \varphi\in C_c(\Omega),
		\]
		where \(\Trace[*]{\A}{\Sigma} := [\Trp{\A}{\Sigma} + \Trm{\A}{\Sigma}]/2\).
\end{itemize}
Moreover, for every $u\in \BVL[\Omega]$, it holds
\begin{itemize}
	\item[(iv)] 
	\(\spair{\A_k, Du} \stackrel{*}{\rightharpoonup} \spair{\A, Du}\)
	locally in the weak${}^*$ sense of measures in \(\Omega\);
	\item[(v)]
	the sequence 
	\(\theta(\A_k, Du; \cdot)\) 
	weakly${}^*$ converges in $L^\infty(\Omega, |Du|)$ to
	\(\theta(\A, Du; \cdot)\),
	where $\theta(\A, Du; \cdot)$ is the Radon--Nikod\'ym derivative of $\spair{\A, Du}$
with respect to $|Du|$.
\end{itemize}
\end{proposition}

\begin{remark}
	It is not difficult to show that a similar approximation result holds also
	for \(\A\in\DMlocloc\) with a sequence \((\A_k)\) in
	\(C^\infty(\Omega,\R^N)\). 
\end{remark}

\begin{proof}
(i) This part is proved in \cite[Theorem 1.2]{ChenFrid}.
We just recall, for later use, that
for every $k$
the vector field \(\A_k\) is of the form
\begin{equation}\label{f:Akk}
\A_k = \sum_{i=1}^\infty \rho_{\eps_i} \ast (\A \varphi_i),
\end{equation}
where \((\varphi_i)\) is a partition of unity subordinate to
a locally finite covering of \(\Omega\)
depending on $k$ and, for every $i$,
\(\eps_i \in (0, 1/k)\) 
is chosen in such a way that
\begin{equation}
\label{f:epsi}
\int_\Omega \left|
\rho_{\eps_i} \ast (\A\cdot\nabla\varphi_i) - \A\cdot \nabla \varphi_i
\right|\, dx\leq \frac{1}{k\, 2^i}
\end{equation}	
(see \cite{ChenFrid}, formula (1.8)).	

\smallskip
(ii) From (i) we have that 
\[
\lim_{k\to +\infty}\int_\Omega \A_k\cdot\nabla\varphi\, dx 
= \int_\Omega \A\cdot\nabla\varphi\, dx
\qquad
\forall \varphi\in C^1_c(\Omega),
\]
hence (ii) follows from  
the density of \(C^1_c(\Omega)\) in \(C_0(\Omega)\)
in the norm of \(L^\infty(\Omega)\)
and the bound
\(\sup_k \int_\Omega |\Div \A_k|\, dx < +\infty\).

\smallskip	
(iii)
As a first step we prove that,
for every $u\in\BVL[\Omega]$,
\begin{equation}
\label{f:ivp0}
\lim_{k\to +\infty}\int_{\Omega} u \, \varphi \, \Div\A_k\, dx
= \int_{\Omega} \prec[*]{u} \,\varphi \, d \Div\A,
\qquad \forall \varphi\in C_c(\Omega).
\end{equation}
Specifically, from the definition \eqref{f:Akk} of $\A_k$ 
and the identity $\sum_i \nabla\varphi_i = 0$
we have that
\[
\Div \A_k = \sum_i \rho_{\eps_i}\ast (\varphi_i \Div\A)
+ \sum_i \left[
\rho_{\eps_i} \ast (\A \cdot \nabla\varphi_i) - \A \cdot \nabla\varphi_i
\right].
\]
From the estimate \eqref{f:epsi} we have that
\[
\left|
\sum_i \int_\Omega u\, \varphi \left[\rho_{\eps_i} \ast (\A \cdot \nabla\varphi_i) - \A \cdot \nabla\varphi_i\right]\, dx
\right| < \frac{1}{k}\, \|\varphi\|_\infty\,\|u\|_\infty\,,
\]
and hence, to prove \eqref{f:ivp0}, it is enough to show that
\begin{equation}
\label{f:limkapp}
\lim_{k\to +\infty}
\sum_i \int_\Omega u\, \varphi \, \rho_{\eps_i} \ast (\varphi_i \Div \A)
=  
\int_{\Omega} \prec[*]{u} \,\varphi \, d \Div\A.
\end{equation}
On the other hand,
\[
\sum_i \int_\Omega u\, \varphi \, \rho_{\eps_i} \ast (\varphi_i \Div \A)
=
\sum_i \int_\Omega \rho_{\eps_i}\ast ( u\, \varphi) \, \varphi_i \, d \Div \A\,,
\]
hence \eqref{f:limkapp} follows by observing that
the functions $\rho_{\eps_i}\ast(u\,\varphi)$
converge pointwise
$\H^{N-1}$--a.e.\ in $\Omega$
to $\prec[*]{u}\varphi$,
so that
\[
\prec[*]{u}\varphi - \sum_i \varphi_i \rho_{\eps_i}\ast(u\, \varphi)
= \sum_i \varphi_i \left[
\prec[*]{u}\varphi -
 \rho_{\eps_i}\ast(u\,\varphi)
\right]
\to 0,
\qquad
\text{$|\Div\A|$-a.e.~in $\Omega$}.
\]

We remark that, as a consequence of \eqref{f:ivp0}, if 
\(E\Subset\Omega\) is a set of finite perimeter,
then
\begin{equation}
\label{f:ivp}
\lim_{k\to +\infty}\int_{\Omega} \chi_{E} \, \varphi \, \Div\A_k\, dx
= \int_{\Omega} \chi^*_{E} \,\varphi \, d \Div\A,
\qquad \forall \varphi\in C_c(\Omega).
\end{equation}

\smallskip	
Let us now prove (iii).
	Let \(\mset\Subset \Omega\) be a set of class \(C^1\).
	By the definition \eqref{f:disttr} of normal traces,
	by (i), (ii) and \eqref{f:ivp},
	for every \(\varphi\in C^\infty_c(\Omega)\) we have that
\[
\begin{split}
	\pscal{\Trace[]{\A_k}{\partial\mset}}{\varphi}
	& = \int_{\mset} \A_k \cdot \nabla\varphi\, dx
	+ \int_{\mset} \varphi \, \Div\A_k\, dx
	\\ & = 
	\int_{\Omega} \chi_{\mset} \, \A_k \cdot \nabla\varphi\, dx
	+ \int_{\Omega} \chi_{\mset} \, \varphi \, \Div\A_k\, dx\,,
\end{split}
\]
so that
\[
\begin{split}
\lim_{k\to +\infty} \pscal{\Trace[]{\A_k}{\partial\mset}}{\varphi}
& =
	\int_{\Omega} \chi_{\mset} \,\A \cdot \nabla\varphi\, dx
	+ \int_{\Omega} \chi^*_{\mset} \,\varphi \, d \Div\A
	\\ & =
	\int_{\mset} \A \cdot \nabla\varphi\, dx
	+ \int_{\mset} \varphi \, d\Div\A
	+ \frac{1}{2} \int_{\partial\mset} \varphi\, d\Div\A
\\ & =
\pscal{\Trace[]{\A}{\partial\mset}}{\varphi}
	+ \frac{1}{2} \pscal{\Div\A\res\partial\mset}{\varphi} 
\,.
	\end{split}
	\]
Hence, by \eqref{f:trA}, we have proved that
\[
\lim_{k\to +\infty} \Trm{\A_k}{\partial\mset}
= \Trm{\A}{\partial\mset}
+\frac{1}{2}\left[\Trp{\A}{\partial\mset}-\Trm{\A}{\partial\mset}\right]
= \Trace[*]{\A}{\partial\mset}\,,
\]
in the sense of distributions.
Using the arguments of Section~\ref{distrtraces}, this relation
can be extended to the countably $\Haus{N-1}$-rectifiable set $\Sigma$.
By a density argument as in (ii), this relation hold for every
$\varphi\in C_c(\Omega)$, hence (iii) holds true for $\Trm{\A_k}{\Sigma}$.
Finally, a similar computation holds for \(\Trp{\A_k}{\Sigma}\).

\smallskip
(iv)	
Using the passage to the limit
in \eqref{f:ivp0} we obtain straightforwardly
\[
\begin{split}
\lim_{k\to +\infty}\pscal{\spair{\A_k, Du}}{\varphi}
& = 
\lim_{k\to +\infty} \left[-\int_\Omega \prec[*]{u}\, \varphi\, \Div\A_k \, dx
- \int_\Omega u\, \A_k\cdot \nabla\varphi\, dx\right]
\\
& =
-\int_\Omega \prec[*]{u}\, \varphi\, d\Div\A
- \int_\Omega u\, \A\cdot \nabla\varphi\, dx
\\
& =
\pscal{\spair{\A, Du}}{\varphi}
\end{split}
\]
for every $\varphi\in C^1_c(\Omega)$.
The validity of this relation for $\varphi\in C_c(\Omega)$
follows from \eqref{f:astim} and the fact that 
the sequence $(\A_k)$ is
bounded in $L^\infty(\Omega, \R^N)$.

\smallskip
(v)
Using the definition \eqref{f:polar} of the density \(\theta\),
we have that, for every \(\varphi\in C_c(\Omega)\),
\[
\begin{split}
\lim_{k\to +\infty}\int_\Omega \theta(\A_k, Du, x) \varphi(x)\, d|Du|
& =
\lim_{k\to +\infty}\pscal{\spair{\A_k, Du}}{\varphi}
\\ & =
\pscal{\spair{\A, Du}}{\varphi}
=
\int_\Omega \theta(\A, Du, x) \varphi(x)\, d|Du|\,.
\end{split}
\]
Since, by \eqref{f:astim} and \eqref{f:polar}, 
the sequence $(\theta(\A_k, Du, \cdot))$ is
bounded in $L^\infty(\Omega, |Du|)$,
then (v) follows.
\end{proof}

\section{Coarea formula for generalized pairings}
\label{s:coarea}

This section is devoted to the proof of the coarea formula
for the $\lambda$-pairing,
and a related slicing result for its density $\theta_\lambda$.

\begin{theorem}[Coarea formula]\label{t:coarea}
Let $\A\in\DMloc[\Omega]$ and let $u\in\BVAloc$.
Then $\chiut\in BV(\Omega)$ for $\LLU$-a.e.\ $t\in\R$, and
	\begin{equation}\label{f:gsplit}
	\pscal{\pair{\A, Du}}{\varphi} = \int_{\R} \pscal{\pair{\A, 
	D\chiut}}{\varphi}\, dt,
	\qquad\forall \varphi\in C_0(\Omega)\,.
	\end{equation}
\end{theorem}

\begin{proof}
Since 
$\pair{\A, Du}$ and $\pair{\A, D\chiut}$
are measures in \(\Omega\) for $\LLU$-a.e.\ $t\in\R$,
it is enough to prove \eqref{f:gsplit} for $\varphi\in C^\infty_c(\Omega)$.

	Let us first consider the case \(u\in L^\infty(\Omega)\).
	By possibly replacing \(u\) with \(u +\|u\|_\infty\),
	it is not restrictive to assume that \(u\geq 0\).
Given a test function \(\varphi\in C^\infty_c(\Omega)\),
we have that
	\begin{equation}
	\label{f:gI1I2}
	\begin{split}
	\int_{\R} & \pscal{\pair{\A, D\chiut}}{\varphi}\, dt  \\
	& =
	-\int_0^{+\infty} \left(
	\int_{\Omega} \pchiut \varphi\, d\Div\A
	\right)\, dt -
	\int_0^{+\infty} \left(
	\int_{\Omega} \chiut \A\cdot\nabla\varphi\, dx
	\right)\, dt \\
	& =  
	-\int_0^{+\infty} \left(
		\int_{\Omega} \pchiut \varphi\, d\Div\A
		\right)\, dt -
		\int_{\Omega} u\, \A\cdot\nabla\varphi\, dx\,.
	\end{split}
	\end{equation}
	
	Moreover, by \cite[Lemma~2.2]{DCFV2}, we have that,
	for \(\mathcal{L}^1\)-a.e.\ \(t\in\R\),
	there exists a Borel set \(N_t\subset\Omega\), with \(\hh(N_t) = 0\),
	such that
\[
\forall x\in \Omega\setminus N_t:
	\qquad
\chiut[-](x) = \chi_{\{u^- > t\}}(x),
\quad
\chiut[+](x) = \chi_{\{u^+ > t\}}(x),
\]
	so that, since \(|\Div\A| \ll \hh\), we obtain that
\[
	\pchiut(x)=(1-\lambda(x)) \chi_{\{u^- > t\}}(x) 
	+\lambda(x)\chi_{\{u^+ > t\}}(x)\,,
	\qquad
	\text{for \(|\Div\A|\)-a.e.}\ x\in\Omega.
\]
	Hence, we get 
	\begin{equation}
	\label{f:gI1}
	\begin{split}
	& \int_0^{+\infty} \left(
			\int_{\Omega} \pchiut \varphi\, d\Div\A
			\right)\, dt \\
	& = 
	\int_0^{+\infty} \left(\int_{\Omega}
	[(1-\lambda) \chi_{\{u^- > t\}} +\lambda\chi_{\{u^+ > 
		t\}}]\, \varphi\, d\Div\A
	\right)\, dt \\ & 
	=
	\int_\Omega (1-\lambda)\varphi \left(\int_0^{+\infty}\chi_{\{u^- > 
	t\}}\, dt\right)\, d\Div\A 
	+\int_\Omega\lambda\varphi\left(\int_0^{+\infty}\chi_{\{u^+ > 
			t\}}\, dt\right)\, d\Div\A \\
	& = \int_\Omega \prec{u} \, \varphi\, d\Div\A\,.
	\end{split}
	\end{equation}
	As a consequence, from \eqref{f:gI1I2}, \eqref{f:gI1}  and
	the definition \eqref{f:gpairing} of \(\pair{\A, Du}\),
	we conclude that \eqref{f:gsplit} holds
	for every test function \(\varphi\in C^\infty_c(\Omega)\)
and for every $u\in\BVL$.

	Finally, the general case $u\in\BVA$ follows
	applying the previous step to the truncated
	functions \(u_k := T_k(u)\).
Specifically, \eqref{f:gsplit} gives, for every $k>0$,
\begin{equation}\label{f:gsplit3}
\pscal{\pair{\A, Du_k}}{\varphi} = \int_{\R} \pscal{\pair{\A, 
D\chi_{\{u_k > t\}}}}{\varphi}\, dt,
\qquad\forall \varphi\in C^1_c(\Omega).
\end{equation}
By Remark~\ref{r:trunc}, the left-hand side of \eqref{f:gsplit3}
converges to $\pscal{\pair{\A, Du}}{\varphi}$.
On the other hand, since
\begin{gather*}
\{u > t\} = \{u_k > t\},\quad
D\chiut = D\chi_{\{u_k > t\}}\,,
\qquad \forall t\in [-k,k),
\\
D\chi_{\{u_k > t\}} = 0\,,
\qquad \forall t \in \R\setminus [-k, k),
\end{gather*}
the right-hand side in \eqref{f:gsplit3} is equal to
\begin{equation}\label{f:intk}
\int_{-k}^k \pscal{\pair{\A, 
D\chiut}}{\varphi}\, dt\,.
\end{equation}
By the estimate \eqref{f:astim} we have that
\[
\left|\pscal{\pair{\A, D\chiut}}{\varphi}\right|
\leq 2 \|\varphi\|_\infty \|\A\|_\infty |D\chiut|(\Omega)\,,
\]
and hence, by the coarea formula in $BV$ and the
Lebesgue Dominated Convergence Theorem,
the integral in \eqref{f:intk} converges to the right-hand side
of \eqref{f:gsplit} as $k\to +\infty$.
\end{proof}

\begin{proposition}\label{l:theta}
Let $\A\in\DMloc[\Omega]$ and  $u\in \BVLloc[\Omega]$.
Then
\begin{equation}
\label{f:gthetaut}
\text{for $\mathcal{L}^1$-a.e.}\ t\in\R:
\quad
\theta_\lambda(\A, Du, x) =
\theta_\lambda(\A, D\chiut, x)
\quad
\text{for \(|D\chiut|\)-a.e.}\ x\in\Omega\,.
\end{equation}
\end{proposition}

\begin{proof}
Thanks to Proposition~\ref{p:approxinf}(iv),
the proof can be done following the lines of
\cite[Proposition~2.7(iii)]{Anz}.
For the reader's convenience, we recall here the main points.

Given two real numbers $a < b$, the function $v := \max\{\min\{u, b\}, a\}$ 
satisfies
\begin{equation}\label{f:uv}
\begin{gathered}
\{u > t\} = \{v > t\},\quad
D\chiut = D\chi_{\{v > t\}}\,,
\qquad \forall t\in [a,b),
\\
D\chi_{\{v > t\}} = 0\,,
\qquad \forall t < a, \  t \geq b.
\end{gathered}
\end{equation}
Since
\[
\frac{d Du}{d|Du|} = \frac{d D\chiut}{d|D\chiut|}\,,
\qquad
\text{$|D\chiut|$-a.e.\ in}\ \Omega
\]
(see \cite[\S 4.1.4, Theorem~2(i)]{GMS1}),
we deduce that
\[
\frac{d Du}{d|Du|}
= \frac{d Dv}{d|Dv|}
\qquad
\text{$|Dv|$-a.e.\ in}\ \Omega.
\]
Let $(\A_k)\subset C^\infty(\Omega,\R^N)\cap L^\infty(\Omega,\R^N)$ be the sequence of
smooth vector fields approximating $\A$ as in
Proposition~\ref{p:approxinf}. 
Since, by \cite[Proposition~2.3]{Anz}, we have
\[
\theta(\A_k, Du, x) =
\A_k(x)\cdot\frac{d Du}{d|Du|}(x)
=
\A_k(x)\cdot\frac{d Dv}{d|Dv|}(x)
=\theta(\A_k, Dv, x)
\qquad
\text{$|Dv|$-a.e.\ in}\ \Omega,
\]
then, from Proposition~\ref{p:approxinf}(v)
and by the uniqueness of the limit in the $L^\infty(\Omega, |Dv|)$
weak${}^*$ topology, we obtain that
\[
\theta(\A, Du, x) = \theta(\A, Dv, x)
\qquad
\text{$|Dv|$-a.e.\ in}\ \Omega.
\]
Recalling the definition \eqref{f:polar} of $\theta_\lambda$
and the relation \eqref{f:resto}, we conclude that
\begin{equation}\label{f:thetauv}
\theta_\lambda(\A, Du, x) = \theta_\lambda(\A, Dv, x)
\qquad
\text{$|Dv|$-a.e.\ in}\ \Omega.
\end{equation}
Specifically, $\theta_\lambda(\A, Du, x) = \theta(\A, Du, x)
= \theta(\A, Dv, x) = \theta_\lambda(\A, Dv, x)$
for $|D^d v|$-a.e.\ $x\in\Omega$,
whereas, by Proposition~\ref{p:pairing}
(and using the notations therein)
and the inclusion $J_v \subset J_u$,
$\theta_\lambda(\A, Du, x) 
= (1-\lambda)\Trp{\A}{J_u} + \lambda\Trm{\A}{J_u}
= \theta_\lambda(\A, Dv, x)$
for $|D^j v|$-a.e.\ $x\in\Omega$.

Given $\varphi\in C^\infty_c(\Omega)$,
let us compute \(\pscal{\pair{\A, Dv}}{\varphi}\).
By the definition of $\theta_\lambda(\A, Dv, x)$, 
equality \eqref{f:thetauv}, the coarea formula in BV
(see \cite[Theorem~3.40]{AFP})
and \eqref{f:uv} it holds
\begin{equation}
\label{f:eqv1}
\begin{split}
\pscal{\pair{\A, Dv}}{\varphi}
& =
\int_\Omega \theta_\lambda(\A, Dv, x)\, \varphi(x)\, d |Dv|
\\ & =
\int_\Omega \theta_\lambda(\A, Du, x)\, \varphi(x)\, d |Dv|
\\ & =
\int_a^b dt
\int_\Omega \theta_\lambda(\A, Du, x)\, \varphi(x)\, d |D\chiut|\,.
\end{split}
\end{equation}
On the other hand, by the coarea formula \eqref{f:gsplit}
and \eqref{f:uv},
it holds
\begin{equation}
\label{f:eqv2}
\begin{split}
\pscal{\pair{\A, Dv}}{\varphi}
& =
\int_\R 
\pscal{\pair{\A, D\chi_{\{v > t\}}}}{\varphi}\, dt
\\ & =
\int_a^b
\pscal{\pair{\A, D\chi_{\{u > t\}}}}{\varphi}\, dt
\\ & =
\int_a^b dt
\int_\Omega \theta_\lambda(\A, D\chiut, x)\, \varphi(x)\, d |D\chiut|\,.
\end{split}
\end{equation}
Comparing \eqref{f:eqv1} with \eqref{f:eqv2},
we finally conclude that, for every $a < b$,
\[
\int_a^b dt
\int_\Omega \theta_\lambda(\A, Du, x)\, \varphi(x)\, d |D\chiut|
=
\int_a^b dt
\int_\Omega \theta_\lambda(\A, D\chiut, x)\, \varphi(x)\, d |D\chiut|\,,
\]
so that \eqref{f:gthetaut} follows.
\end{proof}

\section{Chain rule, Leibniz and Gauss--Green formulas for generalized pairings}
\label{s:chain}

In this section we show that some relevant formulas, proved in \cite{CD3} for
the standard pairing, remain valid for general $\lambda$--pairings.

\begin{proposition}[Chain Rule]\label{p:comp} 
Let $\A\in\DMloc[\Omega]$ and let $u\in \BVLloc[\Omega]$. 
Let $h:\R\to\R$ be a Lipschitz function.
Then it holds:
\begin{itemize}
	\item[(i)]
\(\pair{\A, D[h(u)]}^d = \spair{\A, D[h(u)]}^d\),
and
	\(\pair{\A, D[h(u)]}^a = h'(\ut)\, \A \cdot \nabla u \, \LLN\).
\end{itemize}
Moreover, if \(h\) is non-decreasing, then
\begin{itemize}	
	\item[(ii)]
	\(\displaystyle
	\pair{\A, D[h(u)]}^j = \frac{h(u^+) - h(u^-)}{u^+-u^-} \pair{\A, Du}^j;\)
	\item[(iii)]
	$\theta_\lambda(\A,D[h(u)],x)=\theta_\lambda(\A,Du,x)$,
	for \(|D[h(u)]|\)-a.e.\ $x\in\Omega$.
	
\end{itemize}
The same characterization holds if
$u\in\BVLlocloc$ and
$h\colon I \to\R$ is a locally Lipschitz function
such that $u(\Omega)\Subset I$.
\end{proposition}

\begin{proof}
Although the proof is essentially the same of \cite[Proposition~4.5]{CD3},
for the sake of completeness we prefer to illustrate it in some detail.

One of the main ingredients is 
the Chain Rule Formula for $BV$ functions (see \cite[Theorem~3.99]{AFP}):
\[
D^d [h(u)] = h'(\ut) D^d u,
\quad
D^a [h(u)] = h'(u) \nabla u\, \LLN,
\quad
D^j [h(u)] = [h(u^i) - h(u^e)]\, \nu_u \, \hh \res J_u.
\]
Statement (i) easily follows from the first two relations above
and Proposition~\ref{p:pairing}.
	
Concerning (ii),
we have that \([h(u)]^{i,e} = h(u^{i,e})\)
(see \cite[Proposition~3.69(c)]{AFP}).
Moreover, since $h$ is non-decreasing,
also the relations \([h(u)]^{\pm} = h(u^{\pm})\) hold true,
and hence (ii) follows
again from Proposition~\ref{p:pairing}. 

Let us prove (iii).
If $h$ is strictly increasing, we can follow the proof
of \cite[Proposition~2.8]{Anz}.
Specifically, \(\{u > t\} = \{h(u) > h(t)\}\) for every $t\in\R$, hence
\[
D\chiut = D\chi_{\{h(u) > h(t)\}}\,,
\qquad \forall t\in\R.
\]
From Proposition~\ref{l:theta},
for $\Leb{1}$-a.e.\ $t\in\R$ it holds
\[
\theta_\lambda(\A, Du, x)
=
\theta_\lambda(\A, D\chiut, x)
=
\theta_\lambda(\A, D\chi_{\{h(u) > h(t)\}}, x)
=
\theta_\lambda(\A, D[h(u)], x)
\]
for $|D\chiut|$-a.e.\ $x\in\Omega$,
and (iii) follows.

If $h$ is non-decreasing, we can adapt the proof of
\cite[Proposition~2.7]{LaSe}.
Specifically, let $h_\varepsilon(t) := h(t) + \varepsilon\, t$,
so that $h_\varepsilon$ is strictly increasing for every
$\varepsilon > 0$.
Since
\[
\prec{[h_\varepsilon(u)]}
= (1-\lambda)h_\varepsilon(u^-) + \lambda\, h_\varepsilon(u^+)
= \prec{[h(u)]} + \varepsilon\, \prec{u},
\]
by the previous step we deduce that
\begin{equation}\label{f:peps}
\pair{\A, D[h(u)]} + \varepsilon\, \pair{\A, Du}
=
\pair{\A, D[h_\varepsilon(u)]}
=
\theta_\lambda(\A, Du, x) \, |D[h_\varepsilon(u)]|.
\end{equation}
On the other hand, 
\[
D[h_\varepsilon(u)] =
[h'(\ut) + \varepsilon]\, D^d u
+ [h(u^i)-h(u^e) + \varepsilon (u^i - u^e)]\, D^j u
=
D[h(u)] + \varepsilon\, Du,
\]
hence, passing to the limit
in \eqref{f:peps} as $\varepsilon\to 0$,
we deduce that
\[
\pair{\A, D[h(u)]} = \theta_\lambda(\A, Du, x) \, |D[h(u)]|
\qquad
\text{as measures in $\Omega$},
\]
and (iii) follows.
\end{proof}

\begin{proposition}[Leibniz formula]
\label{p:guA}
Let \(\A\in\DMloc\) 
and \(u,v\in\BVLloc\).
Then, choosing on $J_u$ the orientation
such that  $\upiu=\uint$,
it holds
\begin{gather}
\pair{v\A, Du}^d = \prec{v}\pair{\A, Du}^d = v^* \spair{\A, 
Du}^d,\label{f:gvADud}\\
\pair{v\A, Du}^j 
=  [(1-\lambda)\Trp{\A}{J_u}\vint+\lambda\Trm{\A}{J_u}\vext]\,
(u^+ - u^-) \, \hh\res 
J_u\,.
\label{f:gvADuj}
\end{gather}
\end{proposition}

\begin{proof}
By \cite[Proposition~4.9]{CD3},  
denoting $\alpha^i := \Trp{\A}{J_u}$ and $\alpha^e := \Trm{\A}{J_u}$ 
we have that
\begin{gather}
\spair{v\A,Du}^d = v^* \spair{\A, Du}^d\,,
\label{f:pppd}\\
\spair{v\A,Du}^j =
\frac{\alpha^i\vint+\alpha^e\vext}{2} (\upiu-\umeno)\, \hh\res J_u\,,
\label{f:pppj}
\end{gather}
hence \eqref{f:gvADud} follows from \eqref{f:pppd}
and Proposition~\ref{p:pairing}.

From the representation formulas \eqref{f:trA} and Proposition~\ref{p:tracesb}, we get
\[
\Div(v\A)\res J_u =\left[\Trp{v\A}{J_u}-\Trm{v\A}{J_u} \right] \hh\res 
J_u 
= (\vint\alpha^i-\vext\alpha^e) \hh\res 
J_u\,,
\]
hence, from \eqref{f:pppj},
we obtain
\[
\begin{split}
\pair{v\A,Du}^j = {} &
\left[ 
\frac{\alpha^i\vint+\alpha^e\vext}{2}
+\left(\frac{1}{2}-\lambda\right) 
(\vint\alpha^i-\vext\alpha^e)\right]\,  (\upiu-\umeno)\,\hh\res J_u,
\end{split}
\] 
that is \eqref{f:gvADuj} holds.
\end{proof}

\bigskip
In the last part of this section we will prove a generalized Gauss--Green formula
for vector fields \(\A\in \DMloc[\R^N]\) on a set \(E\subset\R^N\) of finite 
perimeter, generalizing the analogous result for the standard pairing 
proved in \cite[Theorem~5.1]{CD3}.

Using the conventions of Section~\ref{distrtraces},
we will assume that the generalized normal vector on \(\partial^* E\) coincides
\(\hh\)-a.e.\ on \(\partial^* E\) with the measure--theoretic 
\textsl{interior} unit normal vector to \(E\).

\begin{theorem}[Gauss-Green]
\label{t:gGG}
Let $\A \in \DMloc[\R^N]$ and $u\in\BVAloc[\R^N]$.
Let \(E \subset \R^N\) be a bounded set with finite perimeter,
and assume that the traces $\uext,\uint$
of $u$ on $\partial^* E$ belong to $L^1(\partial^* E, \Haus{N-1}\res\partial^* E)$.
Then 
the following Gauss--Green formulas hold:
\begin{gather}
\int_{E^1} \prec{u} \, d\Div\A + \int_{E^1} \pair{\A, Du} = -
\int_{\partial ^*E} \Trp{\A}{\partial^* E}\,\uint \ d\mathcal H^{N-1}\,,\label{gGreenIB}
\\
\int_{E^1\cup \partial ^*E} \prec{u} \, d\Div\A + \int_{E^1\cup \partial ^*E} 
\pair{\A, Du} = -
\int_{\partial ^*E} \Trm{\A}{\partial^* E}\,\uext \ d\mathcal H^{N-1}\,,\label{gGreenIB2}
\end{gather}
where $E^1$ is the measure theoretic interior of $E$ and 
\(\partial^* E\) is oriented
with respect to the interior unit normal vector. 
\end{theorem}

\begin{proof}
We recall that, by Lemma~\ref{l:l1}, 
$\prec{u}\in L^1_{\rm{loc}}(\R^N,|\Div \A|)$.
Recalling \eqref{f:resto}, we have that
\[
\int_{E^1} \pair{\A, Du}= \int_{E^1} \spair{\A, Du}+
\int_{E^1} \left(\frac{1}{2}-\lambda\right)
(\prec[+]{u} - \prec[-]{u})\, d\Div\A\,.
\]
On the other hand, by the definition \eqref{f:pr} of $\prec{u}$,
it holds
\[
\begin{split}
\int_{E^1} \prec{u} \, d\Div\A & = \int_{E^1} u^* d\Div\A
-\int_{E^1} \left(\frac{1}{2}-\lambda\right)(\prec[+]{u} - \prec[-]{u})\, d\Div\A 
\end{split}
\]
so that \eqref{gGreenIB} 
follows from the Gauss--Green formula for the standard pairing
proved in \cite[Theorem~5.1]{CD3}.
The validity of \eqref{gGreenIB2} can be checked in a very similar way.
\end{proof}

\section{Semicontinuity results}
\label{s:sc}

In this section we consider the pairing as a function in $BV$
\[
\BVA[\Omega] \ni u \mapsto \pair{\A,Du} \in \mathcal{M}_b(\Omega),
\]
where $\mathcal{M}_b(\Omega)$ denotes the space of finite Borel measures
on $\Omega$
(see \eqref{f:astim}).

Our aim is to characterize the selections $\lambda\colon\Omega\to [0,1]$
such that 
the above map is lower (resp.\ upper) semicontinuous,
meaning that, if $(u_n)\subset \BVA[\Omega]$ is a sequence converging
to a function $u\in \BVA[\Omega]$ (in a suitable way), then
\begin{gather*}
\pscal{\pair{\A,Du}}{\varphi}\leq\liminf_n \pscal{\pair{\A,Du_n}}{\varphi}
\qquad
\forall\varphi\in C^\infty_c(\Omega),\
\varphi\geq 0
\\
\left(
\text{resp.}\quad
\pscal{\pair{\A,Du}}{\varphi}\geq\limsup_n \pscal{\pair{\A,Du_n}}{\varphi}
\qquad
\forall\varphi\in C^\infty_c(\Omega),\
\varphi\geq 0
\right).
\end{gather*}
Since \(\pair{\A,Du}\) is affected by the pointwise value of
$u$, the correct notion of convergence in $BV$ seems to be
the strict one (see e.g.\ \cite[Definition~3.14]{AFP}).

\begin{definition} 
The sequence $(u_n)\subset BV(\Omega)$ {\sl strictly converges} to $u\in BV(\Omega)$ if $(u_n)$ converges to $u$ in $L^1(\Omega)$ and the 
total variations $|Du_n|(\Omega)$ converge to $|Du|(\Omega)$.
\end{definition}

We recall a recent result concerning the pointwise behavior of strictly
converging sequences.

\begin{proposition}\label{p:lathi}
Every sequence $(u_n)$ strictly convergent in 
$BV(\Omega)$ to $u$ admits a subsequence $(u_{n_k})$
such that for $\hh$-a.e.\ $x\in\Omega$
\begin{equation}\label{f:lahti}
\umeno(x) \leq \liminf_k u_{n_k}^-(x)\leq \limsup_k u_{n_k}^+(x)
\leq \upiu(x).
\end{equation}
In particular, $\lim_k \widetilde{u}_{n_k}(x)=\ut(x)$ for $\hh$-a.e.\ 
$x\in\Omega\setminus J_u$.
\end{proposition}
\begin{proof}
See \cite{La}, Theorem 3.2, and Corollary 3.3.
\end{proof}

Combining Proposition~\ref{p:lathi} with Theorem 3.3 in \cite{CDLP},
we obtain the following approximation result. 

\begin{proposition}
\label{p:approx}
Let $\Omega\subset\R^N$ be a bounded open set,
and let $u\in BV(\Omega)$.
Then there exist two sequences $(u_n), (v_n) \subset W^{1,1}(\Omega)$
such that:
\begin{itemize}
\item [(a)]
for every $n\in\N$,
$\preciso{v}_n \leq\prec[-]{u}$ and $\prec[+]{u} \leq \preciso{u}_n$
$\Haus{N-1}$-a.e.\ in $\Omega$;
\item [(b)]
$u_n\to u$, $v_n\to u$ strictly in $BV$;
\item[(c)]
$\preciso{u}_n(x) \to u^+(x)$ and $\preciso{v}_n(x) \to u^-(x)$
for $\hh$-a.e.\ $x\in\Omega$.
\end{itemize}
If, in addition, $u\in L^\infty(\Omega)$, then the above sequences
are bounded in $L^\infty(\Omega)$.
\end{proposition}

\begin{proof}
From Theorem 3.3 in \cite{CDLP}, there exists a sequence
$(u_n)\subset W^{1,1}(\Omega)$,
strictly convergent to $u$,
and such that $\preciso{u}_n \geq u^+$ 
$\Haus{N-1}$-a.e.\ in $\Omega$, for every $n\in\N$.
Moreover, if $u$ is bounded, then this sequence
is bounded in $L^\infty(\Omega)$.
By Proposition~\ref{p:lathi}, we can extract a subsequence
(not relabeled) such that
\[
\limsup_n \prec[+]{u_n}(x) \leq\prec[+]{u}(x),
\qquad \text{for $\Haus{N-1}$-a.e.}\ x\in\Omega.
\]
On the other hand, the inequality $\preciso{u}_n \geq u^+$
gives
\[
\liminf_n \prec[+]{u_n}(x) \geq\prec[+]{u}(x),
\qquad \text{for $\Haus{N-1}$-a.e.}\ x\in\Omega,
\]
hence the assertion for $(u_n)$ follows.
The construction of $(v_n)$ can be done in a similar way.
\end{proof}

\medskip

In order to state the semicontinuity results,
a more piece of notation is needed.
Given a vector field $\A\in\DM(\Omega)$,
let us denote by $\Omega_{\A}$ the set of points $x\in\Omega$
such that
$x$ belongs to the support of $\Div\A$
(i.e.\ $|\Div\A|(B_r(x)\cap\Omega) > 0$ for every $r>0$),
and the limit
\[
\psi_{\A}(x) := \lim_{r\to 0} \frac{\Div\A (B_r(x))}{|\Div\A|(B_r(x))}
\]
exists in $\R$, with $|\psi_{\A}(x)|=1$.
If we extend $\psi_{\A} = 0$ in $\Omega\setminus\Omega_{\A}$,
we have that $\psi_{\A}\in L^1(\Omega,|\Div \A|)$ and the
polar decomposition $\Div\A = \psi_{\A} |\Div\A|$ holds.
Moreover, if we define the sets
\begin{equation}\label{f:Omegapm}
\Omega^+_{\A} :=
\{x\in\Omega_{\A}:\ \psi_{\A}(x) = 1\},
\quad
\Omega^-_{\A} :=
\{x\in\Omega_{\A}:\ \psi_{\A}(x) = -1\},
\end{equation}
then $(\Div\A)^+ = \Div\A\res\Omega^+_{\A}$
and $(\Div\A)^- = -\Div\A\res\Omega^-_{\A}$.

Let $\Theta_{\A}$ be the jump set of the measure $|\Div\A|$
(see Proposition~\ref{p:basicVF}).
Since $\Theta_{\A}$ is $\sigma$-finite with respect to $\hh$,
then there exists a countably $\hh$-rectifiable Borel set 
$\Theta^r_{\A} \subseteq \Theta_{\A}$
such that $\Theta^u_{\A} := \Theta_{\A} \setminus \Theta^r_{\A}$
is purely $\hh$-unrectifiable 
(i.e.\ $\hh(\Theta^u_{\A} \cap \Sigma) = 0$
for every countably $\hh$-rectifiable set $\Sigma$,
see \cite[Definition~2.64 and Proposition~2.76]{AFP}).

Let us define the families of selections
\begin{gather*}
\Lambda_{{\rm lsc}} :=
\left\{
\lambda\colon \Omega\to [0,1]\ \text{Borel:}\
\lambda = 0\
\text{$\hh$-a.e.\ in}\ \Theta^r_{\A}\cap\Omega^-_{\A},
\lambda = 1\
\text{$\hh$-a.e.\ in}\ \Theta^r_{\A}\cap\Omega^+_{\A}
\right\}
\\
\Lambda_{{\rm usc}} :=
\left\{
\lambda\colon \Omega\to [0,1]\ \text{Borel:}\
\lambda = 1\
\text{$\hh$-a.e.\ in}\ \Theta^r_{\A}\cap\Omega^-_{\A},
\lambda = 0\
\text{$\hh$-a.e.\ in}\ \Theta^r_{\A}\cap\Omega^+_{\A}
\right\}.
\end{gather*}

These families satisfy the following extremality properties.

\begin{lemma}
\label{l:lsc}
Given \(\A\in\DM(\Omega)\), 
$u\in\BVAloc$,
$\varphi\in C_0(\Omega)$, $\varphi\geq 0$,
then for every Borel function $\lambda\colon\Omega\to [0,1]$
it holds
\begin{equation}\label{f:dispm}
\int_{\Omega^{+}_{\A}} \prec{u} \varphi \, d\, \Div \A
\leq
\int_{\Omega^{+}_{\A}} \prec[+]{u} \varphi \, d\, \Div \A\,,
\qquad
\int_{\Omega^{-}_{\A}} \prec{u} \varphi \, d\, \Div \A
\leq
\int_{\Omega^{-}_{\A}} \prec[-]{u} \varphi \, d\, \Div \A\,,
\end{equation}
with equality if $\lambda\in\Lambda_{{\rm lsc}}$.

Similarly,
\begin{equation}\label{f:disusc}
\int_{\Omega^{+}_{\A}} \prec{u} \varphi \, d\, \Div \A
\geq
\int_{\Omega^{+}_{\A}} \prec[-]{u} \varphi \, d\, \Div \A\,,
\qquad
\int_{\Omega^{-}_{\A}} \prec{u} \varphi \, d\, \Div \A
\geq
\int_{\Omega^{-}_{\A}} \prec[+]{u} \varphi \, d\, \Div \A\,,
\end{equation}
with equality if $\lambda\in\Lambda_{{\rm usc}}$.
\end{lemma}

\begin{proof}
Let us prove the claim only for 
the first inequality in \eqref{f:dispm}, the other being similar.

Since, by the very definition of  $\Omega^{+}_{\A}$,
\[
\int_{\Omega^{+}_{\A}} \prec{u} \varphi \, d\, \Div \A
= \int_{\Omega^{+}_{\A}} \prec{u} \varphi \, d\, |\Div \A|
\]
and $\prec{u}\leq\prec[+]{u}$ $\hh$-a.e.\ in $\Omega$,
the first inequality in \eqref{f:dispm} follows.

Let $\lambda\in\Lambda_{{\rm lsc}}$ and let us prove 
that equality holds in the first inequality in \eqref{f:dispm}.
Let us decompose the set $\Omega_{\A}^+$, defined in \eqref{f:Omegapm},
as the union of the disjoint sets
\[
\Omega_{\A}^+ \setminus J_u,
\quad
\Omega_{\A}^+ \cap ( J_u \cap \Theta_{\A}),
\quad
\Omega_{\A}^+ \cap ( J_u \setminus \Theta_{\A}),
\]
that, in turn, coincide up to sets of $\hh$-measure zero
respectively with
\[
\Omega_{\A}^+ \setminus S_u,
\quad
\Omega_{\A}^+  \cap \Theta^r_{\A} \cap J_u,
\quad
(\Omega_{\A}^+\setminus \Theta_{\A})\cap J_u.
\]
Observe that $\prec{u} = \preciso{u}$ 
$\hh$-a.e.\ (hence $|\Div\A|$-a.e.) in $\Omega_{\A}^+ \setminus S_u$,
$\prec{u} = \prec[+]{u}$ 
$\hh$-a.e.\ in $\Omega_{\A}^+  \cap \Theta^r_{\A}$,
and, by Proposition~\ref{p:basicVF},
$|\Div\A|((\Omega_{\A}^+\setminus \Theta_{\A})\cap J_u) = 0$
Hence, 
\begin{equation}\label{f:disp}
\begin{split}
\int_{\Omega^+_{\A}} \prec{u} \varphi \, d\, \Div \A
= {} &
\int_{\Omega^+_{\A}} \prec{u} \varphi \, d\, |\Div \A|
\\ = {} &
\int_{\Omega^+_{\A}\setminus S_u} \preciso{u}\, \varphi \, d\, |\Div \A|
+
\int_{\Omega_{\A}^+ \cap \Theta^r_{\A}\cap J_u } \prec[+]{u}\, \varphi \, d\, |\Div \A|
\\ = {} &
\int_{\Omega^+_{\A}} \prec[+]{u} \varphi \, d\, \Div \A\,.
\qedhere
\end{split}
\end{equation}
\end{proof}

\begin{corollary}
Given \(\A\in\DM(\Omega)\) and $u\in\BVAloc$, it holds:
\begin{gather}
\pair{\A,Du}
= 
-\prec[+]{u} \, (\Div\A)^+
+\prec[-]{u} \, (\Div\A)^-
+ \Div(u\,\A),
\qquad \forall\lambda\in\Lambda_{{\rm lsc}},
\label{f:plsc}
\\
\pair{\A,Du}
= 
-\prec[-]{u} \, (\Div\A)^+
+\prec[+]{u} \, (\Div\A)^-
+ \Div(u\,\A),
\qquad \forall\lambda\in\Lambda_{{\rm usc}}.
\label{f:pusc}
\end{gather}
In particular, 
\begin{gather}
\pair{\A,Du} = \min\{\pair[0]{\A,Du}, \pair[1]{\A,Du}\},
\qquad
\forall\lambda\in\Lambda_{{\rm lsc}},
\label{f:trlsc1}
\\
\pair{\A,Du} = \max\{\pair[0]{\A,Du}, \pair[1]{\A,Du}\},
\qquad
\forall\lambda\in\Lambda_{{\rm usc}}.
\label{f:trusc1}
\end{gather}
Moreover, if the orientation of $J_u$ is chosen in such a way that $\upiu=\uint$,
then, 
\begin{gather}
\pair{\A,Du}^j = \min\{\Trp{\A}{J_u} ,\, \Trm{\A}{J_u}\}
\, (u^+-u^-)\hh\res J_u,
\qquad
\forall\lambda\in\Lambda_{{\rm lsc}}\,,
\label{f:trlsc}
\\
\pair{\A,Du}^j = \max\{\Trp{\A}{J_u} ,\, \Trm{\A}{J_u}\}
\, (u^+-u^-)\hh\res J_u,
\qquad
\forall\lambda\in\Lambda_{{\rm usc}}\,.
\label{f:trusc}
\end{gather}
\end{corollary}

\begin{proof}
The first part is a direct consequence of the equality case
in Lemma~\ref{l:lsc}.

Let us prove \eqref{f:trlsc1}.
To simplify the notation,
let 
\[
\mu := \Div\A, \qquad
\nu := \min\{\pair[0]{\A,Du}, \pair[1]{\A,Du}\}.
\]
Since
$\pair[0]{\A,Du} = -u^- \mu + \Div(u\A)$
and $\pair[1]{\A,Du} = -u^+ \mu + \Div(u\A)$,
by definition of minimum of two measures, for every Borel set 
$E\subset\Omega$ one has
\[
\nu(E) = \Div(u\A)(E) +
\inf\left\{
-u^- \mu^+(E_0) - u^+ \mu^+(E_1) 
+ u^- \mu^-(E_0) + u^+ \mu^-(E_1)
\right\},
\] 
where the infimum is taken over the
pairs $E_0, E_1$ of disjoint Borel sets
such that $E = E_0\cup E_1$.
Setting $E^- :=  E\cap\Omega_{\A}^-$
and $E^+ := E \setminus E^-$,
then $E\cap\Omega_{\A}^+ \subset E^+$ and
\begin{gather*}
-u^- \mu^+(E_0) - u^+ \mu^+(E_1) \geq
-u^+\mu^+(E^+) = -u^+\mu^+(E), 
\\
u^- \mu^-(E_0) + u^+ \mu^-(E_1) \geq
u^-\mu^-(E^-) = u^-\mu^-(E), 
\end{gather*}
for every partition $\{E_0, E_1\}$ of $E$.
Hence,
\[
\nu(E) = \Div(u\A)(E) -u^+\mu^+(E) + u^-\mu^-(E)
= \pair{\A, Du}(E),
\qquad
\forall\lambda\in\Lambda_{{\rm lsc}}\,.
\]
The proof of \eqref{f:trusc1} is similar.
Finally,
\eqref{f:trlsc} and \eqref{f:trusc}
are consequences of \eqref{f:trlsc1} and \eqref{f:trusc1},
respectively, and Proposition~\ref{p:pairing}.
\end{proof}

\bigskip

\begin{theorem}
\label{t:lsc}
Let \(\A\in\DM(\Omega)\),
and let $\lambda\colon \Omega \to [0,1]$ be a Borel function.

Then $\lambda\in \Lambda_{{\rm lsc}}$ if and only if,
for every $u_n, u\in BV(\Omega)$ satisfying
\begin{itemize}
\item[(a)]
$u_n \to u$ strictly in $BV$,
\item[(b)] there exists $g\in L^1(\Omega, |\Div\A|)$
such that, for every $n\in\N$,  $|u_n^\pm| \leq g$ $|\Div\A|$-a.e.\
in $\Omega$, 
\end{itemize}
it holds
\begin{equation}\label{f:lsc}
\pscal{\pair{\A,Du}}{\varphi}
\leq
\liminf_n \pscal{\pair{\A,Du_n}}{\varphi}
\qquad
\forall \varphi\in C^\infty_c(\Omega),
\
\varphi\geq 0.
\end{equation}
Analogously, 
$\lambda\in \Lambda_{{\rm usc}}$ if and only if,
for every $u_n, u\in BV(\Omega)$ satisfying (a), (b)
it holds
\begin{equation}\label{f:usc}
\pscal{\pair{\A,Du}}{\varphi}
\geq
\limsup_n \pscal{\pair{\A,Du_n}}{\varphi}
\qquad
\forall \varphi\in C^\infty_c(\Omega),
\
\varphi\geq 0.
\end{equation} 
\end{theorem}

\begin{proof}
Let us prove only the statement concerning the lower semicontinuity,
the other being similar.

Let $\lambda\in \Lambda_{{\rm lsc}}$,
let $u_n, u\in BV(\Omega)$ satisfy (a), (b),
and let us prove that the semicontinuity property in \eqref{f:lsc} holds.
Let $\varphi\in C^\infty_c(\Omega)$, $\varphi\geq 0$,
and let  $(u_{n_k})$ be a subsequence such that 
\[
\liminf_n \pscal{\pair{\A,Du_n}}{\varphi}=\lim_k 
\pscal{\pair{\A,Du_{n_k}}}{\varphi},
\]
and \eqref{f:lahti} holds true
(here we use (a) and Proposition~\ref{p:lathi}).

From Lemma~\ref{l:lsc}, assumption (b), Fatou's Lemma and 
the pointwise estimates \eqref{f:lahti} we have that
\begin{equation}\label{f:limsup1}
\limsup_k\int_{\Omega_{\A}^+} \prec{u_{n_k}}\, \varphi\, d\Div\A
\leq
\limsup_k\int_{\Omega_{\A}^+} \prec[+]{u_{n_k}}\, \varphi\, d\Div\A
\leq
\int_{\Omega_{\A}^+} \prec[+]{u}\, \varphi\, d\Div\A\,.
\end{equation}
Recalling that
\[
\int_{\Omega_{\A}^-} \prec{u_{n_k}}\, \varphi\, d\Div\A =
-\int_{\Omega_{\A}^-} \prec{u_{n_k}}\, \varphi\, d|\Div\A|\,,
\]
the same argument gives
\begin{equation}\label{f:limsup2}
\limsup_k\int_{\Omega_{\A}^-} \prec{u_{n_k}}\, \varphi\, d\Div\A
\leq
\int_{\Omega_{\A}^-} \prec[-]{u}\, \varphi\, d\Div\A\,.
\end{equation}
Since $|\Div\A|(\Omega\setminus(\Omega_{\A}^- \cup \Omega_{\A}^+)) = 0$,
from \eqref{f:limsup1}, \eqref{f:limsup2} and
the equality case in \eqref{f:dispm} 
we get
\begin{equation}\label{f:limsup3}
\limsup_k\int_{\Omega} \prec{u_{n_k}}\, \varphi\, d\Div\A
\leq
\int_{\Omega_{\A}^+} \prec[+]{u}\, \varphi\, d\Div\A
+
\int_{\Omega_{\A}^-} \prec[-]{u}\, \varphi\, d\Div\A
= \int_\Omega \prec{u}\, \varphi\, d\Div\A\,.
\end{equation}
Finally, 
from \eqref{f:limsup3} and (a)
we conclude that
\[
\begin{split}
& \liminf_n \pscal{\pair{\A,Du_n}}{\varphi} =
-\limsup_k\left(
\int_\Omega \prec{u_{n_k}}\, \varphi\, d\Div A
+ \int_\Omega u_{n_k}\, \A\cdot \nabla\varphi\, dx
\right)
\\ & 
\geq
-\int_{\Omega} \prec{u}\, \varphi\, d\Div A
- \int_\Omega u\, \A\cdot \nabla\varphi\, dx
\\ & =
\pscal{\pair{\A,Du}}{\varphi}\,,
\end{split}
\]
i.e., \eqref{f:lsc} holds true.

\bigskip
Assume now that 
\eqref{f:lsc} holds true
for every $u_n, u\in BV(\Omega)$ satisfying (a), (b),
and let us prove that $\lambda\in\Lambda_{{\rm lsc}}$.
We claim that, under these assumptions,
\begin{equation}
\label{f:ineq}
\pair{\A,Du}\leq \pair[0]{\A, Du}\,,
\quad
\pair{\A,Du}\leq \pair[1]{\A, Du}\,,
\qquad
\forall u\in \BVA,
\end{equation}
in the sense of measures.
By a truncation argument and Remark~\ref{r:trunc}, it is enough to show that
the above inequality holds for every $u\in BV(\Omega)\cap L^\infty(\Omega)$.
Let $u\in BV(\Omega)\cap L^\infty(\Omega)$
and let $(u_n), (v_n) \subset W^{1,1}(\Omega)\cap L^\infty(\Omega)$
be the approximating sequences given by Proposition~\ref{p:approx}.
Observe that these sequences are bounded in $L^\infty(\Omega)$, so that
they satisfy assumption (b).
Since $(\preciso{u}_n)$ converges to $\prec[+]{u}$ $|\Div\A|$-a.e.\ in $\Omega$
and, by (b), also in $L^1(\Omega, |\Div\A|)$,
for every test function $\varphi\in C^\infty_c(\Omega)$ we have that
\[
\begin{split}
\lim_n\pscal{\pair{\A, Du_n}}{\varphi}
& = \lim_n \left(-\int_\Omega \preciso{u}_n \, \varphi\, d\Div\A 
- \int_\Omega u_n\, \A \cdot \nabla\varphi\, dx\right)
\\ & =
-\int_\Omega \prec[+]{u} \, \varphi\, d\Div\A 
- \int_\Omega u\, \A \cdot \nabla\varphi\, dx
= \pscal{\pair[1]{\A, Du}}{\varphi},
\end{split}
\]
hence, by the semicontinuity assumption, if $\varphi\geq 0$,
\[
\pscal{\pair{\A, Du}}{\varphi}
\leq
\liminf_n \pscal{\pair{\A, Du_n}}{\varphi}
=
\pscal{\pair[1]{\A, Du}}{\varphi}\,.
\]
The same argument, using the sequence $(v_n)$, shows that
\[
\pscal{\pair{\A, Du}}{\varphi}
\leq
\liminf_n \pscal{\pair{\A, Dv_n}}{\varphi}
=
\pscal{\pair[0]{\A, Du}}{\varphi}\,,
\]
so that \eqref{f:ineq} follows.

Let $\Omega' \Subset \Omega$ be an open domain with $C^1$ boundary.
From Proposition~\ref{p:pairing} we have that
\[
\pair{\A, D\chi_{\Omega'}} =
\left[(1-\lambda) \Trp{\A}{\partial\Omega'} +
\lambda\, \Trm{\A}{\partial\Omega'}\right]\, \hh\res \partial\Omega',
\]
hence, the inequalities \eqref{f:ineq} give
\begin{equation}
\label{f:tri0}
\begin{cases}
(1-\lambda) \left[\Trp{\A}{\partial\Omega'} - \Trm{\A}{\partial\Omega'}\right] \leq 0,
\\
-\lambda\, \left[\Trp{\A}{\partial\Omega'} - \Trm{\A}{\partial\Omega'}\right] \leq 0,
\end{cases}
\qquad
\text{$\hh$-a.e.\ on}\ \partial\Omega'.
\end{equation}
Let $\Sigma\subset\Omega$ be an oriented countably $\hh$-rectifiable set.
Recalling the definition of normal traces given in
Section~\ref{distrtraces},
from \eqref{f:tri0} we deduce that
\begin{equation}\label{f:tri}
\begin{cases}
(1-\lambda) \left[\Trp{\A}{\Sigma} - \Trm{\A}{\Sigma}\right] \leq 0,
\\
-\lambda\, \left[\Trp{\A}{\Sigma} - \Trm{\A}{\Sigma}\right] \leq 0,
\end{cases}
\qquad
\text{$\hh$-a.e.\ on}\ \Sigma.
\end{equation}
Let us choose an orientation for the countably $\hh$-rectifiable set
$\Sigma^+ := \Theta_{\A}^r \cap \Omega_{\A}^+$.
Since 
$\Sigma^+ \subset \Omega_{\A}$
and $\psi_{\A}(x) = 1$ for $|\Div\A|$-a.e.\ $x\in\Sigma^+$, 
from \eqref{f:trA} we have that
\[
\Div\A \res \Sigma^+ =
\left[\Trp{\A}{\Sigma^+} - \Trm{\A}{\Sigma^+}\right]
\hh\res \Sigma^+ > 0.
\]
Hence,
from the first inequality in \eqref{f:tri},
we deduce that $\lambda = 1$ $\hh$-a.e.\ on $\Sigma^+$.
A similar argument, using $\Sigma^- := \Theta_{\A}^r \cap \Omega_{\A}^-$,
shows that $\lambda = 0$ $\hh$-a.e.\ on $\Sigma^-$. 
\end{proof}

\begin{corollary}
\label{c:cont}
Let \(\A\in\DM(\Omega)\)
and let $\lambda\colon\Omega\to [0,1]$ be a Borel function.
Then the continuity property
\begin{equation}\label{f:cont}
\pscal{\pair{\A,Du}}{\varphi}
=
\lim_n \pscal{\pair{\A,Du_n}}{\varphi}
\qquad
\forall \varphi\in C^\infty_c(\Omega),
\end{equation}
holds for every $u_n, u\in BV(\Omega)$ satisfying (a) and (b)
in Theorem~\ref{t:lsc}
if and only if 
$\hh(\Theta_{\A}^r) = 0$.
\end{corollary}

\begin{proof}
We have that the stated property holds
if and only if 
both \eqref{f:lsc} and \eqref{f:usc} hold.
From Theorem~\ref{t:lsc},
these inequalities hold (for every $(u_n)$, $u$)
if and only if $\lambda \in \Lambda_{{\rm lsc}} \cap \Lambda_{{\rm usc}}$.
Finally, from the very definition of
$\Lambda_{{\rm lsc}}$ and $\Lambda_{{\rm usc}}$,
we have that
$\Lambda_{{\rm lsc}} \cap \Lambda_{{\rm usc}}\neq \emptyset$
if and only if $\hh(\Theta_{\A}^r) = 0$.
\end{proof}

\begin{remark}
The assumption $\hh(\Theta_{\A}^r) = 0$
is trivially satisfied if $\Div^j\A = 0$,
e.g.\ if $\Div\A \in L^1(\Omega)$.
\end{remark}

\begin{example}
In view of Corollary~\ref{c:cont} we have that,
in general, 
the continuity property \eqref{f:cont} does not hold
with respect to the strict convergence in $BV$.
Specifically, let $\Omega = (-2,2)\subset\R$ and consider
$\A := \chi_{(-1,1)}$,
so that $\Div\A = \delta_{-1} - \delta_{1}$,
and $\Theta_{\A}^r = \Theta_{\A} = \{-1, +1\}$
is not empty.
Let $\lambda\colon \Omega \to [0,1]$ be any Borel function.
Let $u_n(x) := \max\{\min\{n+1-n|x|, 1\}, 0\}$.
It is readily seen that $(u_n)$ strictly converges to
$u := \chi_{[-1,1]}$, so that $(-u_n)$ strictly converges
to $-u$, and
$\pscal{\pair{\A, Du_n}}{\varphi} = 0$ for every $n$.
On the other hand, 
choosing $\varphi$ such that $\varphi(-1) = 0$ and
$\varphi(1) = 1$,
one has
\begin{gather*}
\pscal{\pair{\A, Du}}{\varphi}
= [1-\lambda(-1)] \, \varphi(-1)
+ [\lambda(1)-1]\, \varphi(1) = \lambda(1) - 1,
\\
\pscal{\pair{\A, D(-u)}}{\varphi}
= -\lambda(-1) \, \varphi(-1)
+ \lambda(1)\, \varphi(1) = \lambda(1),
\end{gather*}
and at least one of the right-hand sides must be different
from $0$.
\end{example}

\begin{example}
We remark that, in general, 
\eqref{f:lsc} does not hold
if assumption (a) is replaced by
the weak${}^*$ convergence in $BV$.
Specifically, let us consider $\Omega = (-2,2)\subset\R$,
$\A := \chi_{(0,1)}$ and
$u_n(x) := \max\{1-n|x|, 0\}$.
Since $(u_n)$ converges to $u=0$ in $L^1(\Omega)$ and
$|Du_n|(\Omega) = 2$ for every $n$,
then 
$(u_n)$ converges weakly${}^*$ to $u$ in $BV(\Omega)$.
(see \cite[Proposition~3.13]{AFP}):
If $\varphi\in C^\infty_c(\Omega)$
is strictly positive in $0$, one has
\[
\begin{split}
\liminf_n
\pscal{\pair{\A, Du_n}}{\varphi}
& =
\liminf_n
\left(
-u_n(0)\, \varphi(0) + u_n(1)\, \varphi(1) -\int_0^1 u_n\, \varphi'
\right)
\\ & = -\varphi(0) 
< 0 = \pscal{\pair{\A, Du}}{\varphi}\,,
\end{split}
\]
so that \eqref{f:lsc} does not hold.
\end{example}

\bigskip
\noindent
{\bf Acknowledgments.}
The authors would like to thank 
Giovanni E.\ Comi
for some useful remarks
on a preliminary version of the manuscript.
A.M.\ and V.D.C.\ 
have been partially supported by the Gruppo Nazionale per l'Analisi Matematica, 
la Probabilit\`a e le loro Applicazioni (GNAMPA) of the Istituto Nazionale di Alta Matematica (INdAM). 
G.C.\ and A.M.\ have been partially supported by Sapienza - Ateneo 2017 Project "Differential Models in Mathematical Physics".

\def\cprime{$'$}

\end{document}